\documentclass[12pt,letterpaper]{amsbook} \usepackage{amssymb}
\usepackage{amsfonts} \usepackage{amsmath} \usepackage{amsthm} \usepackage{bm}
\usepackage{latexsym} \usepackage{epsfig}
\usepackage{hyperref}
\hypersetup{
  colorlinks=true,
  linkcolor=blue,
  citecolor=red
}
\usepackage{mathrsfs}
\usepackage{bbm}
\numberwithin{equation}{chapter}
\newtheorem*{theorem*}{Theorem}
\newtheorem{theorem}{Theorem}[section]
\newtheorem{lemma}[theorem]{Lemma}
\newtheorem*{proposition*}{Proposition}

\newtheorem{proposition}[theorem]{Proposition}
\newtheorem{claim}[theorem]{Claim}
\newtheorem{corollary}[theorem]{Corollary}

\newtheorem{assumption}[theorem]{Assumption}
\newtheorem{example}[theorem]{Example}

\DeclareMathOperator{\argmax}{\operatornamewithlimits{argmax}}

\renewcommand{\P}[1]{{\mathbb{P}}\left[{#1}\right]}
\newcommand{\Psub}[2]{{\mathbb{P}}_{#1}\left[{#2}\right]}
\newcommand{\CondP}[2]{{\mathbb{P}}\left[{#1}\middle\vert{#2}\right]}
\newcommand{\E}[1]{{\mathbb{E}}\left[{#1}\right]}

\newcommand{\CondE}[2]{{\mathbb{E}}\left[{#1}\middle\vert{#2}\right]}
\newcommand{\ind}[1]{\mathbbm{1}({#1})}
\newcommand{\Var}[1]{{\mathrm{Var}}\left[{#1}\right]}

\newcommand{\eps}{\epsilon}

\renewcommand{\phi}{\varphi}

\newcommand{\N}{\mathbb N}
\newcommand{\R}{\mathbb R}

\newcommand{\sgn}{\mathrm{sgn}}

\newcommand{\half}{{\textstyle \frac12}}

\newcommand{\InfGraphs}{\mathcal{B}}
\newcommand{\GraphFam}{\mathcal{A}}

\newcommand{\belief}{B}
\newcommand{\fbelief}{B}

\newcommand{\llr}{Z} %private log likelihood ratio
\newcommand{\psignal}{W} %private signal
\newcommand{\action}{A}
\newcommand{\estL}{K}
\newcommand{\logit}{\mathrm{logit}}

\newcommand{\util}{U}
\newcommand{\neigh}[1]{\partial{#1}}

\newcommand{\info}{H}

\newcommand{\cM}{\mathcal{M}}

\newcommand{\cG}{\mathcal{G}}

\newcommand{\cK}{\mathcal{K}}
\newcommand{\cF}{\mathcal{F}}

\newcommand{\round}[1]{\mathrm{round}\left({#1}\right)}
\newcommand{\dtv}{d_{\mathrm{TV}}}
\newcommand{\dkl}[2]{d_{\mathrm{KL}}\left({#1}\middle\vert{#2}\right)}

\title{Opinion Exchange Dynamics}
\author{Elchanan Mossel and Omer Tamuz\\{\small Copyright 2014. All rights reserved to
    the authors}}
\begin{document}
\maketitle

\pagenumbering{roman}
\tableofcontents
\chapter{Introduction}
\pagenumbering{arabic}
\setcounter{page}{1}

\section{Modeling opinion exchange}
The exchange of opinions between individuals is a fundamental social
interaction that plays a role in nearly any social, political and
economic process. While it is unlikely that a simple mathematical
model can accurately describe the exchange of opinions between two
people, one could hope to gain some insights on emergent phenomena
that affect large groups of people.

Moreover, many models in this field are an excellent playground for
mathematicians, especially those working in probability, algorithms
and combinatorics. The goal of this survey is to introduce such models
to mathematicians, and especially to those working in discrete
mathematics, information theory, optimization, probability and
statistics.

%A particular phenomenon that we shall see arising across diverse
%settings is that {\em decentralization} or {\em egalitarianism}
%promote the free flow of information, while the concentration of power
%may inhibit it. 

\subsection{Modeling approaches}

Many of the models we discuss in the survey comes from the literature in theoretical economics. 
In microeconomic theory, the main paradigm of modeling human
interaction is by a {\em game}, in which participants are rational
agents, choosing their moves optimally and responding to the
strategies of their peers. A particularly interesting class of games
is that of probabilistic {\em Bayesian games}, in which players also
take into account the uncertainty and randomness of the world. We
study Bayesian models in Section~\ref{chap:bayesian}.

Another class of models, which have a more explicit combinatorial
description, are what we refer to as {\em heuristic} models. These
consider the dynamics that emerge when agents are assumed to utilize
some (usually simple) update rule or algorithm when interacting with
each other. Economists often justify such models as describing agents
with {\em bounded rationality}. We study such models in
Section~\ref{chap:heuristic}.

It is interesting that both of these approaches are often justified by
an {\em Occam's razor} argument. To justify the heuristic models,
the argument is that assuming that people use a simple heuristic
satisfies Occam's razor. Indeed, it is undeniable that the simpler the
heuristic, the weaker the assumption. On the other hand, the Bayesian
argument is that even by choosing a simple heuristic one has too much
freedom to reverse engineer any desired result. Bayesians therefore
opt to only assume that agents are rational. This, however, may result
in extremely complicated behavior.

There exists several other natural dichotomies and sub-dichotomies. In
rational models, one can assume that agents tell each other their
opinions. A more common assumption in Economics is that agents learn
by observing each other's {\em actions}; these are choices that an
individual makes that not only reflect their belief, but also carry
potential gain or penalty. For example, in financial markets one could
assume that traders tell each other their value estimates, but a
perhaps more natural setting is that they learn about these values by
seeing which actual bids their peers place, since the latter are
costly to manipulate. Hence the adage ``actions speak louder than
words.''

Some actions can be more revealing than others. A bid by a trader
could reveal the value the trader believes the asset carries, but in a
different setting it could perhaps just reveal whether the trader
thinks that the asset is currently overpriced or underpriced. In other
models an action could perhaps reveal {\em all} that an agent
knows. We shall see that widely disparate outcomes can result in
models that differ only by how revealing the actions are.

Although the distinction between {\em opinions}, {\em beliefs} and
{\em actions} is sometimes blurry, we shall follow the convention of
having agents learn from each other's actions. While in some models
this will only be a matter of nomenclature, in others this will prove
to be a pivotal choice. The term {\bf belief} will be reserved for a
technical definition (see below), and we shall not use {\em opinion},
except informally.

\section{Mathematical Connections} 
Many of the models of information exchange on networks are intimately
related to nice mathematical concepts, often coming from probability, discrete mathematics, optimization and information theory. 
We will see how the theories of Markov
chains, martingale arguments, influences and graph limits all play a
crucial role in analyzing the models we describe in these notes. Some
of the arguments and models we present may fit well as classroom
materials or exercises in a graduate course in probability.

\section{Related Literature} 
It is impossible to cover the huge body of work related to information
exchange in networks.  We will cite some relevant papers at each
section. Mathematicians reading the economics literature may benefit
from keeping the following two comments in mind:
\begin{itemize}
\item The focus in economics is not the mathematics, but the
  economics, and in particular the justification of the model and the
  interpretation of the results. Thus important papers may contain
  little or no new mathematics. Of course, many papers do contain
  interesting mathematics.
\item For mathematicians who are used to models coming from natural
  sciences, the models in the economics literature will often look like
  very rough approximation and the conclusions drawn in terms of real
  life networks unjustified. Our view is that the models have very
  limited implication towards real life and can serve as most as
  allegories. We refer the readers who are interested in this point to
  Rubinstein's book ``Economic Fables''~\cite{Rubinstein:12}. 
\end{itemize}

\section{Framework}
\label{sec:framework}

The majority of models we consider share the a common underlying
framework, which describes a set of agents, a state of the world, and
the information the agents have regarding this state. We describe it
formally in Section~\ref{sec:definitions} below, and shall note
explicitly whenever we depart from it.

We will take a probabilistic / statistical point of view in studying
models.  In particular we will assume that the model includes a random
variable $S$ which is the {\em true state of the world}. It is this
$S$ that all agents want to learn.  For some of the models, and in
particular the rational, economic models, this is a natural and even
necessary modeling choice. For some other models - the voter model,
for example (Section~\ref{sec:voter-models}), this is a somewhat
artificial choice. However, it helps us take a single perspective by
asking, for each model, how well it performs as a statistical
procedure aimed at estimating $S$.  Somewhat surprisingly, we will
reach similarly flavored conclusions in widely differing settings. In
particular, a repeated phenomenon that we observe is that {\em
  egalitarianism}, or {\em decentralization} facilitates the flow of
information in social networks, in both game-theoretical and
heuristic models.

\section{General definitions}
\label{sec:definitions}
\subsection{Agents, state of the world and private signals}
Let $V$ be a countable set of agents, which we take to be
$\{1,2,\ldots,n\}$ in the finite case and $\N=\{1,2,\ldots\}$ in the
infinite case.  Let $\{0,1\}$ be the set of possible values of the
{\em state of the world} $S$.

Let $\Omega$ be a compact metric space equipped with the Borel
sigma-algebra. For example, and without much loss of generality,
$\Omega$ could be taken to equal the closed interval $[0,1]$. Let
$\psignal_i \in \Omega$ be agent $i$'s private signal, and denote
$\bar{\psignal} = (\psignal_1,\psignal_2,\ldots)$.

Fix $\mu_0$ and $\mu_1$, two mutually absolutely continuous measures
on $\Omega$. We assume that $S$ is distributed uniformly, and that
conditioned on $S$, the $\psignal_i$'s are i.i.d.\ $\mu_S$: when $S=0$ then
$\bar{\psignal} \sim \mu_0^V$, and when $S=1$ then $\bar{\psignal}
\sim \mu_1^V$.

More formally, let $\delta_0$ and $\delta_1$ be the distributions on
$\{0,1\}$ such that $\delta_0(0)=\delta_1(1)=1$. We consider the
probability space $\{0,1\} \times \Omega^V$, with the measure
$\mathbb{P}$ defined by
\begin{align*}
  \mathbb{P} = \mathbb{P}_{\mu_0,\mu_1,V} =
  \half\delta_0\times\mu_0^V+\half\delta_1\times\mu_1^V,
\end{align*}
and let
\begin{align*}
  (S,\bar{\psignal}) \sim \mathbb{P}.
\end{align*}

\subsection{The social network}
A {\bf social network} $G=(V,E)$ is a directed graph, with $V$ the set
of agents.  The set of neighbors of $i \in V$ is $\neigh{i} = \{j :\:
(i,j) \in E\} \cup \{i\}$ (i.e., $\neigh{i}$ includes $i$). The {\bf
  out-degree} of $i$ is given by $|\neigh{i}|$. The {\bf degree} of
$G$ is give by $\sup_{i \in V}|\neigh{i}|$.

We make the following assumption on $G$.
\begin{assumption}
  We assume throughout that $G$ is simple and strongly connected, and
  that each out-degree is finite.
\end{assumption}

We recall that a graph is strongly connected if for every two nodes
$i,j$ there exists a directed path from $i$ to $j$.  Finite
out-degrees mean that an agent observes the actions of a finite number
of other agents. We do allow infinite {\bf in-degrees}; this
corresponds to agents whose actions are observed by infinitely many
other agents. In the different models that we consider we impose
various other constraints on the social network.

\subsection{Time periods and actions}
We consider the discrete time periods $t=0,1,2,\ldots$, where in each
period each agent $i \in V$ has to choose an action $\action^i_t \in
\{0,1\}$. This action is a function of agent $i$'s private signal, as
well as the actions of its neighbors in previous time periods, and so
can be thought of as a function from
$\Omega\times\{0,1\}^{|\neigh{i}|\cdot t}$ to $\{0,1\}$. The exact
functional dependence varies among the models.

\subsection{Extensions, generalizations, variations and special cases}
The framework presented above admits some natural extensions,
generalizations and variations. Conversely, some special cases deserve
particular attention. Indeed, some of the results we describe apply
more generally, while others do not apply more generally, or apply
only to special cases. We discuss these matters when describing each
model.

\begin{itemize}
\item {\bf The state of the world} can take values from sets larger
  than $\{0,1\}$, including larger finite sets, countably infinite
  sets or continuums.
\item {\bf The agents' private signals} may not be i.i.d.\ conditioned
  on $S$: they may be independent but not identical, they may be
  identical but not independent, or they may have a general joint
  distribution.

  An interesting special case is when the space of private signals is
  equal to the space of the states of the world. In this case one can
  think of the private signals as each agent's initial guess of $S$.
\item A number of models consider only undirected {\bf social
    networks}, that is, symmetric social networks in which $(i,j) \in
  E \Leftrightarrow (j,i) \in E$.
\item More general network model include weighted directed models where different directed edges
have different weights. 
 
\item {\bf Time} can be continuous. In this case we assume that each
  agent is equipped with an i.i.d.\ {\em Poisson clock} according to
  which it ``wakes up'' and acts. In the finite case this is
  equivalent to having a single, uniformly chosen   random agent act
  in each discrete time period. It is also possible to define more general continuous time processes. 
\item {\bf Actions} can take more values than $\{0,1\}$. In particular
  we shall consider the case that actions take values in $[0,1]$.

  In order to model randomized behavior of the agents, we shall also
  consider actions that are not measurable in the private signal, but
  depend also on some additional randomness. This will require the
  appropriate extension of the measure $\mathbb{P}$ to a larger
  probability space.
\end{itemize}

\section{Questions}
\label{sec:questions}
 The main phenomena that we shall study are {\em convergence}, {\em
  agreement}, {\em unanimity}, {\em learning} and more.

\begin{itemize}
\item {\bf Convergence.}  We say that agent $i$ {\em converges} when
  $\lim_t\action^i_t$ exists. We say that the entire process converges
  when all agents converge.

  The question of convergence will arise in all the models we study,
  and its answer in the positive will often be a requirement for
  subsequent treatment. When we do have convergence we define
  \begin{align*}
    \action^i_\infty = \lim_{t \to \infty}\action^i_t.
  \end{align*}
  
\item {\bf Agreement and unanimity.} We say that agents $i$ and $j$
  {\em agree} when $\lim_t\action^i_t=\lim_t\action^j_t$. {\em
    Unanimity} is the event that $i$ and $j$ agree for all pairs of
  agents $i$ and $j$. In this case we can define
  \begin{align*}
    \action_\infty = \action^i_\infty,
  \end{align*}
  where the choice of $i$ on the r.h.s.\ is immaterial.

\item {\bf Learning.} We say that agent $i$ {\em learns} $S$ when
  $\action^i_\infty = S$, and that learning occurs in a model when all
  agents learn. In cases where we allow actions in $[0,1]$, we will
  say that $i$ learns whenever $\round{\action^i_\infty}=S$, where
  $\round{\cdot}$ denotes rounding to the nearest integer, with
  $\round{1/2} = 1/2$.

  We will also explore the notion of {\em asymptotic learning}. This
  is said to occur for a sequence of graph $\{G_n\}_{n=1}^\infty$ if
  the agents on $G_n$ learn with probability approaching one as $n$
  tends to infinity.
  
\end{itemize}

A recurring theme will be the relation between these questions and the
{\em geometry} or {\em topology} of the social network. We shall see
that indeed different networks may exhibit different behaviors in
these regards, and that in particular, and across very different
settings, {\em decentralized} or {\em egalitarian} networks tend to
promote {\em learning}.

\section{Acknowledgments} 
Allan Sly is our main collaborator in this field. We are grateful to
him for allowing us to include some of our joint results, as well as
for all that we learned from him. The manuscript was prepared for the
9\textsuperscript{th} Probability Summer School in Cornell, which took
place in July 2013. We are grateful to Laurent Saloff-Coste and Lionel
Levine for organizing the school and for the participants for helpful
comments and discussions. We would like to thank Shachar Kariv for
introducing us to this field, and Eilon Solan for encouraging us to
continue working in it.  The research of Elchanan Mossel is partially
supported by NSF grants DMS 1106999 and CCF 1320105, and by ONR grant
N000141110140. Omer Tamuz was supported by a Google Europe Fellowship
in Social Computing.

\chapter{Heuristic Models} 
\label{chap:heuristic}
\section{The DeGroot model}
The first model we describe was pioneered by Morris DeGroot in
1974~\cite{DeGroot:74}. DeGroot's contribution was to take standard
results in the theory of Markov Processes (See, e.g.,
Doob~\cite{doob1953stochastic}) and apply them in the social setting.
The basic idea for these models is that people repeatedly average
their neighbors' actions. This model has been studied extensively in
the economics literature.  The question of {\em learning} in this
model has been studied by Golub and Jackson~\cite{golub2010naive}.
%general.

%{\bf DeGroot Learning must have been discussed before. I am not sure I like this reference here ...}

\subsection{Definition}
Following our general framework (Section~\ref{sec:definitions}), we
shall consider a state of the world $S \in \{0,1\}$ with conditionally
i.i.d.\ private signals. The distribution of private signals is what
we shall henceforth refer to as {\em Bernoulli} private signals: for
some $\half > \delta > 0$, $\mu_i(S) = \half + \delta$ and $\mu_i(1-S)
= \half - \delta$, for $i=0,1$.  Obviously this is equivalent to
setting $\CondP{\psignal_i=S}{S} = \half+\delta$.
%(where, to remind the reader,
%$\psignal_i$ is $i$'s private signal), and that the events
%`$\psignal_i=S$' are independent. Hence an equivalent definition is to
%set $\psignal_i=S$ with probability $\half+\delta$ and
%$\psignal_i=1-S$ with probability $\half-\delta$, and to do so
%independently for all the agents.

In the DeGroot model, we let the actions take
values in $[0,1]$. In particular, we define the actions as follows:
\begin{align*}
  \action^i_0 = \psignal_i
\end{align*}
and for $t>0$
\begin{align}
  \label{eq:degroot-update}
  \action^i_t = \sum_{j \in \neigh{i}}w(i,j)\action^j_{t-1},
\end{align}
where we make the following three assumptions:
\begin{enumerate}
\item $\sum_{j \in \neigh{i}}w(i,j) = 1$ for all $i \in V$.
\item $i \in \neigh{i}$ for all $i \in V$.
\item $w(i,j) > 0$ for all $(i,j) \in E$.
\end{enumerate}
The last two assumptions are non-standard, and, in fact, not strictly
necessary. We make them to facilitate the presentation of the results
for this model.

We assume that the social network $G$ is finite. We consider both the
general case of a directed strongly connected network, and the special
case of an undirected network. 

\subsection{Questions and answers}
We shall ask, with regards to the DeGroot model, the same three
questions that appear in Section~\ref{sec:questions}.
\begin{enumerate}
\item {\bf Convergence.} Is it the case that agents' actions converge?
  That is, does, for each agent $i$, the limit $\lim_t\action^i_t$
  exist almost surely? We shall show that this is indeed the case.
\item {\bf Agreement.} Do all agents eventually reach agreement? That
  is, does $\action^i_\infty=\action^j_\infty$ for all $(i,j) \in V$?
  Again, we answer this question in the positive.
\item {\bf Learning.} Do all agents learn? In the case of continuous
  actions we say that agent $i$ has learned $S$ if
  $\round{\action^i_\infty} = S$. Since we have agreement in this
  model, it follows that either all agents learn or all do not
  learn. We will show that the answer to this question depends on the
  topology of the social network, and that, in particular, a certain
  form of {\em egalitarianism} is a sufficient condition for learning
  with high probability.
\end{enumerate}

\subsection{Results}
The key to the analysis of the DeGroot model is the realization that
\eqref{eq:degroot-update} describes a transformation from the
actions at time $t-1$ to the actions at time $t$ that is the {\em
  Markov operator} $P_w$ of the a random walk on the graph $G$.
However, while usually the analysis of random walks deals with action
of $P_w$ on distributions from the right, here we act on functions
from the left~\cite{doob2001classical}. While this is an important
difference, it is still easy to derive properties of the DeGroot
process from the theory of Markov chains (see, e.g.,
Doob~\cite{doob1953stochastic}).
%  if we
%think of $\action^j_{t-1}$ as the probability that a random walker is
%at node $j$ at time $t-1$, then, by \eqref{eq:degroot-update},
%$\action^i_t$ describes the probability of finding it at node $i$ at
%time $t$. 

Note first, that assumptions (2) and (3) on \eqref{eq:degroot-update} make
this Markov chain irreducible and a-periodic.
Since, for a node $j$ 
\[
\action^j_t = \E{\psignal_{X^j_t}},
\]
where $X^j_t$ is the Markov chain started at $j$ and run for $t$
steps, if follows that $\action^j_\infty :=\lim_t\action^j_t$ is
nothing but the expected value of the private signals, according to
the stationary distribution of the chain.  We thus obtain
\begin{theorem}[Convergence and agreement in the DeGroot model]
  \label{thm:degroot-converge}
  For each $j \in V$, 
\[  \action_\infty :=\lim_t\action^j_t = \sum_{i \in V}\alpha_i\psignal_i,
\] 
where $\alpha = (\alpha_1,\ldots,\alpha_n)$
  is the stationary distribution of the Markov chain described by
  $P_w$.
\end{theorem}
Recall that $\alpha$ is the left eigenvector of $P_w$
corresponding to eigenvalue $1$, normalized in $\ell^1$.  
In the internet age, the vector $\alpha$ is also known as the {\em
  PageRank} vector~\cite{page1999pagerank}. It is the asymptotic
probability of finding a random walker at a given node after
infinitely many steps of the random walk. Note that $\alpha$ is not
random; it is fixed and depends only on the weights $w$. Note also that
Theorem~\ref{thm:degroot-converge} holds for {\em any} real valued
starting actions, and not just ones picked from the distribution
described above.  To gain some insight into the result, let us
consider the case of undirected graphs and simple (lazy) random
walks. For these, it can be shown that
\begin{align*}
  \alpha_i = \frac{|\neigh{i}|}{\sum_j|\neigh{j}|}.
\end{align*}

Recall that $\P{A^i_0=S} = \half + \delta$. We observe the following.
\begin{proposition}[Learning in the DeGroot model]
  For a set of weights $w$, let $p_w(\delta) =
  \P{\round{\action_\infty} = S}$. Then:
  \begin{itemize}
  \item $p_w$ is a monotone function of $\delta$ with $p_w(0) = 1/2$
    and $p_w(1/2) = 1$.
  \item For a fixed $0 < \delta < 1/2$, among all $w$'s on graphs of size $n$, 
    $p_w(\delta)$ is maximized when the stationary distribution of $G$
    is uniform.
  \end{itemize}
\end{proposition}

\begin{proof}
\begin{itemize}
\item The first part follows by coupling. Note that we can couple the
  processes with $\delta_1 < \delta_2$ such that the value is $S$
  is the same and moreover, whenever $\psignal_i = S$ in the
  $\delta_1$ process we also have $\psignal_1 = S$ in the $\delta_2$
  process.  Now, since the vector $\alpha$ is independent of $\delta$
  and $\action_\infty = \sum_i\alpha_i\psignal_i$, the coupling above
  results in $|\action_\infty-S|$ being smaller in the $\delta_2$
  process than it is in the $\delta_1$ process.
\item 
  The second part follows from the Neyman-Peason lemma in statistics. 
  This lemma states that among {\em all} possible estimators, the one that maximizes the probability that $S$ is reconstructed 
  correctly is given by 
  \begin{align*}
    \hat{S} = \round{\frac{1}{n}\sum_i\psignal_i}
  \end{align*}
\end{itemize}   
%  is the maximum likelihood estimator of $S$, given the private
%  signals.
\end{proof}

We note that an upper bound on $p_w(\delta)$ can be obtained using
Hoeffding's inequality~\cite{hoeffding1963probability}. We leave this
as an exercise to the reader.

Finally, the following proposition is again a consequence of well
known results on Markov chains. See the books by
Saloff-Coste~\cite{salof-coste1997lectures} or Levin, Peres and
Wilmer~\cite{levin2009markov} for basic definitions.
\begin{proposition}[Rate of Convergence in the Degroot Model]
 Suppose that at time $t$, the total variation distance between the chain started at $i$ and run for $t$ steps and the stationary distribution is 
 at most $\eps$. Then a.s.: 
\[
  \max_i | \action^i_t - \action_\infty | \leq 2 \eps \delta. 
  \] 
\end{proposition}

\begin{proof}
Note that 
\[
 \action^i_i - \action_\infty = \E{ \psignal_{X^i_t} - \psignal_{X_{\infty}}}. 
 \]
 Since we can couple the distributions of $X_t$ and $X_{\infty}$ so
 that they dis-agree with probability at most $\eps$ and the maximal
 difference between any two private signals is at most $\delta$, the
 proof follows.

\end{proof}

\subsection{Degroot with  cheaters and bribes} 

A cheater is an agent who plays a fixed action. 
\begin{itemize}
\item {\bf Exercise.} Consider the DeGroot model with a single cheater
  who picks some fixed action. What does the process converge to?
\item {\bf Exercise.} Consider the DeGroot model with $k$ cheaters,
  each with some (perhaps different) fixed action. What does the model
  converge to?
\item {\bf Research problem.}  Consider the following zero sum
  game.\footnote{We do not formally define games here. A good
    introduction is Osborne and Rubinstein's
    textbook~\cite{osborne1994course}.} $A$ and $B$ are two
  companies. Each company's strategy is a choice of $k$ cheaters
  (cheaters chosen by both play honestly), for whom the company can
  choose a fixed value in $[0,1]$. The utility of company $A$ is the
  sum of the players' limit actions, and the utility of company $B$ is
  minus the utility of $A$. What are the equilibria of this game?
\end{itemize}

\subsection{The case of infinite graphs}
Consider the DeGroot model on an infinite graph, with a simple random
walk.
\begin{itemize}
\item {\bf Easy exercise.} Give an example of specific  private signals for which 
the limit $\action_\infty$ doesn't exist. 
\item {\bf Easy exercise.} Prove that $\action_\infty$ exists and is equal
  to $S$ on non-amenable graphs a.s. A graph is non-amenable if the
  Markov operator $P_w \colon \ell^2(V) \to \ell^2(V)$ has norm
  strictly less than 1.
\item {\bf Harder exercise.} Prove that $\action_\infty$ exists and is
  equal to $S$ on general infinite graphs.

\end{itemize}

\section{The voter model}
\label{sec:voter-models}

This model was described by P.\ Clifford and A.\
Sudbury~\cite{clifford1973model} in the context of a spatial conflict
where animals fight over territory (1973) and further analyzed by R.A.\
Holley and T.M.\ Liggett~\cite{holley1975ergodic}.

\subsection{Definition}
As in the DeGroot model above, we shall consider a state of the world
$S \in \{0,1\}$ with conditionally i.i.d.\ {\em Bernoulli} private
signals, so that $\P{\psignal_i=S} = \half+\delta$.

We consider binary actions and define them in a way that resembles our
definition of the DeGroot model. We let:
\begin{align*}
  \action^i_0 = \psignal_i
\end{align*}
and for $t>0$, all $i$ and all $j \in \neigh{i}$, 
\begin{align}
  \label{eq:voter-update}
  \P{\action^i_t = \action^j_{t-1}} = w(i,j), 
\end{align}
so that in each round each agent chooses a neighboring agent to
emulate. We make the following assumptions:
\begin{enumerate}
\item All choices are independent. 
\item $\sum_{j \in \neigh{i}}w(i,j) = 1$ for all $i \in V$.
\item $i \in \neigh{i}$ for all $i \in V$.
\item $w(i,j) > 0$ for all $(i,j) \in E$.
\end{enumerate}
As in the DeGroot model, the last two assumptions are non-standard,
and are made to facilitate the presentation of the results for this
model.

We assume that the social network $G$ is finite. We consider both the
general case of a directed strongly connected network, and the special
case of an undirected network. 

\subsection{Questions and answers}
We shall ask, with regards to the voter model, the same three
questions that appear in Section~\ref{sec:questions}.
\begin{enumerate}
\item {\bf Convergence.} Does, for each agent $i$, the limit
  $\lim_t\action^i_t$ exist almost surely? We shall show that this is
  indeed the case.
\item {\bf Agreement.} Does $\action^i_\infty=\action^j_\infty$ for
  all $(i,j) \in V$?  Again, we answer this question in the positive.
\item {\bf Learning.}  In the case of discrete actions we say that
  agent $i$ has learned $S$ if $\action^i_\infty = S$. Since we have
  agreement in this model, it follows that either all agents learn or
  all do not learn. Unlike other models we have discussed, we will show
  that the answer here is no.  Even for large egalitarian networks,
  learning doesn't necessarily holds. We will later discuss a variant
  of the voter model where learning holds.
\end{enumerate}

\subsection{Results}
We first note that 
\begin{proposition}
In the voter model with assumptions~(\ref{eq:voter-update}) all agents converge to the same action.
\end{proposition}

\begin{proof}
  The voter model is a Markov chain. Clearly the states where
  $\action^i_t = 0$ for all $i$ and the state where $\action^i_t = 1$
  for all $i$ are absorbing states of the chain.  Moreover, it is easy
  to see that for any other state, there is a sequence of moves of the
  chain, each occurring with positive probability, that lead to the
  all $0$ / all $1$ state. From this it follows that the chain will
  always converge to either the all $0$ or all $1$ state.
\end{proof} 

We next wish to ask what is the probability that the agents learned $S$? For the voter model this chance is never very high 
as the following proposition shows: 

\begin{theorem}[(Non) Learning in the Voter model]
  \label{thm:voter-learning}
  Let $\action_{\infty}$ denote the limit action for all the agents in the voter model. Then:
  \begin{equation} \label{eq:voter1}
  \P{\action_\infty = 1 | \psignal} = \sum_{i \in V}\alpha_i\psignal_i,
  \end{equation}
  and 
  \begin{equation} \label{eq:voterS}
  \P{\action_\infty = S | \psignal} = \sum_{i \in V}\alpha_i1(\psignal_i = S). 
  \end{equation}
  where $\alpha = (\alpha_1,\ldots,\alpha_n)$
  is the stationary distribution of the Markov chain described by
  $P_w$. Moreover, 
  \begin{equation} \label{eq:voterall}
  \P{\action_\infty = S} = \frac{1}{2} + \delta.
  \end{equation}
\end{theorem}

\begin{proof}
  Note that \eqref{eq:voterS} follows immediately from
  \eqref{eq:voter1} and that \eqref{eq:voterall} follows from
  \eqref{eq:voterS} by taking expectation over $\psignal$. To prove
  \eqref{eq:voter1} we build upon a connection to the DeGroot model.
  Let $D^i_t$ denote the action of agent $i$ in the DeGroot model at
  time $t$. We are assuming that the DeGroot model is defined using
  the same $w(i,j)$ and that the private signals are identical for the
  voter and DeGroot model. Under these assumption it is easy to verify
  by induction on $i$ and $t$ that
\[
\P{\action^i_t = 1} = D^i_t.
\]
Thus 
\[
\P{\action^i_\infty = 1} = D^i_{\infty} = \sum_{i \in V}\alpha_i\psignal_i,
\]
as needed.
\end{proof} 
In the next section we will discuss a variant of the voter model that does lead to learning. 

We next briefly discuss the question of the convergence rate of the
voter model. Here again the connection to the Markov chain of the
DeGroot model is paramount (see, e.g., Holley and
Liggett~\cite{holley1975ergodic}). We will not discuss this beautiful
theory in detail. Instead, we will just discuss the case of undirected
graphs where all the weights are $1$.

{\bf Exercise.} Consider the voter model on an undirected graph with
$n$ vertices. This is equivalent to letting $w(i,j) = 1/d_i$ for all
$i$, where $d_i = |\partial i|$.
\begin{itemize}
\item
Show that $X_t = \sum d_i \action^i_t$ is a martingale. 
\item 
Let $T$ be the stopping time where $\action^i_t = 0$ for all $i$ or $\action^i_t = 1$ for all $i$. 
Show that $\E{X_T} = \E{X_0}$ and use this to deduce that 
\[
\P{\action_\infty = 1 | \psignal} = \frac{\sum_{i \in V} d_i \psignal_i}{\sum_{i \in V} d_i}
\]
\item 
Let $d = \max_i d_i$. Show that 
\[
\E{(X_t - X_{t-1})^2 | t < T} \geq 1/(2d).
\]
Use this to conclude that 
\[
\E{T}/ (2d) \leq \E{(X_T - X_0)^2} \leq n^2,
\]
so 
\[ 
\E{T} \leq 2 d n^2.
\]
\end{itemize}   

\subsection{A variant of the voter model}
As we just saw, the voter model does not lead to learning even on
large egalitarian networks.  It is natural to ask if there are
variants of the model that do. We will now describe such a variant 
(see e.g.~\cite{Benezit:09,MosselSchoenebeck:10}).  For simplicity we
consider an undirected graph $G = (V,E)$ and the following asynchronous
dynamics.
\begin{itemize}
\item 
At time $t=0$, let $\action_i^0 = (\psignal_i,1)$.
\item 
At each time $t \geq 1$ choose an edge $e=(i,j)$ of the graph at random and continue as follows: 
\item 
For all $k \notin \{i,j\}$, let $\action^t_k = \action^{t-1}_k$. 
\item 
Denote $(a_i,w_i) = \action^{t-1}_i$ and $(a_j,w_j) = \action^{t-1}_j$.
\item 
If $a_i \neq a_j$ and $w_i = w_j = 1$, let $a_i' = a_i, a_j' = a_j$ and $w_i' = w_j' = 0$. 
\item 
If $a_i \neq a_j$ and $w_i = 1 > w_j = 0$, let $a_i'= a_j' = a_i$ and $w_i' = w_i$ and $w_j' = w_j$. 
\item 
Similarly, if $a_i \neq a_j$ and $w_j = 1 > w_i = 0$, let $a_i'= a_j' = a_j$ and $w_i' = w_i$ and $w_j' = w_j$. 
\item 
if $a_i \neq a_j$ and $w_j =  w_i = 0$, let $a_i' = a_j' = 0 (1)$ with probability $1/2$ each. Let $w_i' = w_j' = 0$. 
\item 
Otherwise, if $a_i = a_j$, let $a_i' = a_i, a_j'  = a_j, w_i = w_i', w_j = w_j'$. 
\item With probability $1/2$ let $\action^t_i := (a_i',w_i')$ and
  $\action^t_j := (a_j', w_j')$. With probability $1/2$ let
  $\action^t_i := (a_j',w_j')$ and $\action^t_j := (a_i', w_i')$
\end{itemize} 

Here is a useful way to think about this dynamics. The $n$ players all
begin with opinions given by $\psignal_i$.  Moreover these opinions
are all strong (this is indicated by the second coordinate of the
action being $1$).  At each round a random edge is chosen and the two
agents sharing the edge declare their opinions regarding $S$.  If
their opinions are identical, then nothing changes except that with
probability $1/2$ the agents swap their location on the edge.  If the
opinions regarding $S$ differ and one agent is strong (second
coordinate is $1$) while the second one is weak (second coordinate is
$0$) then the weak agent is convinced by the strong agent. If the two
agents are strong, then they keep their opinion but become weak. If
the two of them are weak, then they both choose the same opinion at
random. At the end of the exchange, the agents again swap their
positions with probability $1/2$. We leave the following as an
exercise:

\begin{proposition}
Let $\action^t_i = (X^t_i,Y^t_i)$. Then a.s. 
\[
\lim X^t_i = X,
\]
where 
\begin{itemize}
\item
$X = 1$ if $\sum_i \psignal_i > n/2$, 
\item
$X = 0$ if $\sum_i \psignal_i < n/2$ and 
\item
$\P{X = 1} = 1/2$ if $\sum_i \psignal_i = n/2$.
\end{itemize} 
\end{proposition} 

Thus this variant of the voter model yields optimal learning.

\section{Deterministic iterated dynamics}

A natural deterministic model of discrete opinion exchange dynamics is
{\bf majority dynamics}, in which each agent adopts, at each time
period, the opinion of the majority of its neighbors. This is a model
that has been studied since the 1940's in such diverse fields as
biophysics~\cite{mcculloch1943logical},
psychology~\cite{cartwright1956structural} and
combinatorics~\cite{GO:80}.

\subsection{Definition}

In this section, let $\action^i_0$ take values in $\{-1,+1\}$, and let
\begin{align*}
  \action^i_{t+1} = \sgn\sum_{j \in \neigh{i}}\action^j_t.
\end{align*}
we assume that $|\neigh{i}|$ is odd, so that there are never cases of
indifference and $\action^i_t \in \{-1,+1\}$ for all $t$ and $i$. We
assume also that the graph is undirected.

A classical combinatorial result (that has been discovered
independently repeatedly; see discussion and generalization
in~\cite{GO:80}) is the following.
\begin{theorem}
  \label{thm:finite-period-two}
  Let $G=(V,E)$ be a finite undirected graph. Then
  \begin{align*}
    \action^i_{t+1} = \action^i_{t-1}
  \end{align*}
  for all $i$, for all $t \geq |E|$, and for all initial opinion sets
  $\{\action^j_0\}_{j \in V}$.
\end{theorem}
That is, each agent (and therefore the entire dynamical system)
eventually enters a cycle of period at most two. We prove this below.

A similar result applies to some infinite graphs, as discovered by
Moran~\cite{moran1995period} and Ginosar and
Holzman~\cite{ginosar2000majority}; see also~\cite{tamuz2015majority,
  benjamini2016convergence} . Given an agent $i$, let $n_r(G,i)$ be
the number of agents at distance exactly $r$ from $i$ in $G$. Let
$\mathfrak{g}(G)$ denote the asymptotic growth rate of $G$ given by
\begin{align*}
  \mathfrak{g}(G) = \limsup_rn_r(G,i)^{1/n}.
\end{align*}
This can be shown to indeed be independent of $i$. Then
\begin{theorem}[Ginosar and Holzman, Moran]
  \label{thm:infinite-period-two}
  If $G$ has degree at most $d$ and $\mathfrak{g}(G) < \frac{d+1}{d-1}$
  then for each initial opinion set $\{\action^j_0\}_{j \in V}$ and for
  each $i \in V$ there exists a time $T_i$ such that
  \begin{align*}
    \action^i_{t+1} = \action^i_{t-1}
  \end{align*}
  for all $t \geq T_i$.
\end{theorem}
That is, each agent (but {\em not} the entire dynamical system)
eventually enters a cycle of period at most two. We will not give a
proof of this theorem.

In the case of graphs satisfying $\mathfrak{g}(G) < (d+1)/(d-1)$, and
in particular in finite graphs, we shall denote
\begin{align*}
  \action^i_\infty = \lim_t\action^i_{2t}.
\end{align*}
This exists surely, by Theorem~\ref{thm:infinite-period-two} above.

In this model we shall consider a state of the world $S \in \{-1,+1\}$
with conditionally i.i.d.\ Bernoulli private signals in $\{-1,+1\}$,
so that $\P{\psignal_i=S} = \half+\delta$. As above, we set
$\action^i_0 = \psignal_i$.

\subsection{Questions and answers}
We ask the usual questions with regards to this model.
\begin{enumerate}
\item {\bf Convergence.} While it is easy to show that agents'
  opinions do not necessarily converge in the usual sense, they do
  converge to sequences of period at most two. Hence we will consider
  the limit action $\action^i_\infty = \lim_t\action^i_{2t}$ as
  defined above to be the action that agent $i$ converges to.

\item {\bf Agreement.} This is easily {\em not} the case in this model
  that $\action^i_\infty = \action^j_\infty$ for all $i,j \in
  V$. However, in~\cite{mossel2012majority} it is shown that agreement
  is reached, with high probability, for good enough expander
  graph.\footnote{We do not define expander graphs formally here;
    informally, they are graphs that resemble random graphs.}

\item {\bf Learning.} Since we do not have agreement in this model, we
  will consider a different notion of learning. This notion may
  actually be better described as {\em retention of information}. We
  define it below.  Condorcet's Jury Theorem~\cite{Condorcet:85}, in
  an early version of the law of large numbers, states that given $n$
  conditionally i.i.d.\ private signals, one can estimate $S$
  correctly, except with probability that tends to zero with $n$. The
  question of retention of information asks whether this still holds
  when we introduce correlations ``naturally'' by the process of
  majority dynamics.

  Let $G$ be finite, undirected graphs. Let
  \begin{align*}
    \hat{S} = \argmax_{s \in
      \{-1,+1\}}\CondP{S=s}{\action^1_\infty,\ldots,\action^{|V|}_\infty}.
  \end{align*}
  This is the maximum a-posteriori (MAP) estimator of $S$, given the
  limit actions. Let
  \begin{align*}
    \iota(G,\delta) = \P{\hat{S} \neq S},
  \end{align*}
  where $G$ and $\delta$ appear implicitly in the right hand
  side. This is the probability that the best possible estimator of
  $S$, given the limit actions, is {\em not} equal to $S$.

  Finally, let $\{G_n\}_{n \in \N}$ be a sequence of finite,
  undirected graphs. We say that we have {\em retention of
    information} on the sequence $\{G_n\}$ if $\iota(G_n,\delta) \to_n
  0$ for all $\delta > 0$. This definition was first introduced, to
  the best of our knowledge, in Mossel, Neeman and
  Tamuz~\cite{mossel2012majority}.
 
  Is information retained on all sequences of growing graphs? The
  answer, as we show below, is no. However, we show that information
  is retained on sequences of {\em transitive
    graphs}~\cite{mossel2012majority}.

\end{enumerate}

\subsection{Convergence}

To prove convergence to period at most two for finite graphs, we
define the {\em Lyapunov} functional
\begin{align*}
  L_t = \sum_{(i,j) \in E}(\action^i_{t+1}-\action^j_t)^2.
\end{align*}
We prove Theorem~\ref{thm:finite-period-two} by showing that $L_t$ is
monotone decreasing, that $\action^i_{t+1}=\action^i_{t-1}$ whenever
$L_t-L_{t-1}=0$, and that $L_t=L_{t-1}$ for all $t > |E|$. This proof
appears (for a more general setting) in Goles and
Olivos~\cite{GO:80}. For this we will require the following definitions:
\begin{align*}
  J^i_t = \left(\action^i_{t+1}-\action^i_{t-1}\right)\sum_{j \in
    \neigh{i}}\action^j_t
\end{align*}
and
\begin{align*}
  J_t = \sum_{i \in V}J^i_t.
\end{align*}
\begin{claim}
  $J^i_t \geq 0$ and $J^i_t =0$ iff $\action^i_{t+1}=\action^i_{t-1}$.
\end{claim}
\begin{proof}
  This follows immediately from the facts that
  \begin{align*}
    \action^i_{t+1} = \sgn \sum_{j \in \neigh{i}}\action^j_t,
  \end{align*}
  and that $\sum_{j \in \neigh{i}}\action^j_t$ is never zero.
\end{proof}

It follows that
\begin{corollary}
  \label{cor:period-two}
  $J_t \geq 0$ and $J_t =0$ iff $\action^i_{t+1}=\action^i_{t-1}$ for
  all $i \in V$. 
\end{corollary}

We next show that $L_t$ is monotone decreasing.
\begin{proposition}
  \label{prop:J-L}
  $L_t-L_{t-1} = -J_t$.
\end{proposition}
\begin{proof}
  By definition,
  \begin{align*}
    L_t - L_{t-1} = \sum_{(i,j) \in E}(\action^i_{t+1}-\action^j_t)^2
    - \sum_{(i,j) \in E}(\action^i_t-\action^j_{t-1})^2.
  \end{align*}
  Opening the parentheses and canceling identical terms yields
  \begin{align*}
    L_t - L_{t-1} = -2\sum_{(i,j) \in E}\action^i_{t+1}\action^j_t
    +2 \sum_{(i,j) \in E}\action^i_t\action^j_{t-1}.    
  \end{align*}
  Since the graph is undirected we can change variable on the right
  sum and arrive at
  \begin{align*}
    L_t - L_{t-1} &= -2\sum_{(i,j) \in
      E}\action^i_{t+1}\action^j_t-\action^j_t\action^i_{t-1}\\
    &=  -2\sum_{(i,j) \in
      E}\left(\action^i_{t+1}-\action^i_{t-1}\right)\action^j_t.
  \end{align*}
  Finally, applying the definitions of $J^i_t$ and $J_t$ yields
  \begin{align*}
    L_t - L_{t-1} = -\sum_{i \in V}J^i_t = -J_t.
  \end{align*}
\end{proof}

\begin{proof}[Proof of Theorem~\ref{thm:finite-period-two}]
  Since $L_0 \leq |E|$, $L_t \leq L_{t-1}$ and $L_t$ is integer, it
  follows that $L_t \neq L_{t-1}$ at most $|E|$ times. Hence, by
  Proposition~\ref{prop:J-L}, $J_t >0$ at most $|E|$ times. But if
  $J_t=0$, then the state of the system at time $t+1$ is the same as
  it was at time $t-1$, and so it has entered a cycle of length at
  most two. Hence $J_t =0$ for all $t > |E|$, and the claim follows.
\end{proof}

\subsection{Retention of information}
In this section we prove that
\begin{enumerate}
\item There exists a sequence of finite, undirected graphs $\{G_n\}_{n
    \in \N}$ of size tending to infinity such that $\iota(G,\delta)$
  does not tend to zero for any $0 < \delta < \half$.
\item Let $\{G_n\}_{n \in \N}$ be a sequence of finite, undirected,
  connected {\em transitive} graphs of size tending to infinity. Then
  $\iota(G_n,\delta) \to_n 0$, and, furthermore, if we let $G_n$ have
  $n$ vertices, then
  \begin{align*}
    \iota(G_n,\delta) \leq Cn^{-\frac{C\delta}{\log(1/\delta)}}.
  \end{align*}
  for some {\em universal} constant $C > 0$.
\end{enumerate}
A {\em transitive graph} is a graph for which, for every two vertices
$i$ and $j$ there exists a {\em graph homomorphism} $\sigma$ such that
$\sigma(i) =j$. A graph homomorphism $h$ is a permutation on the vertices
such that $(i,j) \in E$ iff $(\sigma(i),\sigma(j)) \in
E$. Equivalently, the group $\mathrm{Aut}(G) \leq S_{|V|}$ acts
transitively on $V$.

Berger~\cite{berger2001dynamic} gives a sequence of graphs $\{H_n\}_{n
  \in \N}$ with size tending to infinity, and with the following
property. In each $H_n=(V,E)$ there is a subset of vertices $W$ of
size $18$ such that if $\action^i_t=-1$ for some $t$ and all $i \in W$
then $\action^j_\infty=-1$ for all $j \in V$. That is, if all the
vertices in $W$ share the same opinion, then eventually all agents
acquire that opinion.
\begin{proposition}
  $\iota(H_n,\delta) \geq (1-\delta)^{18}$.
\end{proposition}
\begin{proof}
  With probability $(1-\delta)^{18}$ we have that $\action^i_0 = -S$
  for all $i \in W$. Hence $\action^j_\infty=-S$ for all $j \in V$,
  with probability at least $(1-\delta)^{18}$. Since the MAP estimator
  $\hat{S}$ can be shown to be a symmetric and monotone function of
  $\action^j_\infty$, it follows that in this case $\hat{S}=-S$, and so
  \begin{align*}
    \iota(H_n,\delta) = \P{\hat{S} \neq S} \geq (1-\delta)^{18}.
  \end{align*}
\end{proof}

We next turn to prove the following result
\begin{theorem}
  \label{thm:transitive-learning}
  Let $G$ a finite, undirected, connected {\em transitive} graph with
  $n$ vertices, $n$ odd. then
  \begin{align*}
    \iota(G,\delta) \leq Cn^{-\frac{C\delta}{\log(1/\delta)}}.
  \end{align*}
  for some {\em universal} constant $C > 0$.
\end{theorem}

Let $\hat{S} = \sgn\sum_{i \in V}\action^i_\infty$ be the result of a
majority vote on the limit actions. Since $n$ is odd then $\hat{S}$
takes values in $\{-1,+1\}$. Note that $\hat{S}$ is measurable in the
initial private signals $\psignal_i$. Hence there exists a function $f
: \{-1,+1\}^n \to \{-1,+1\}$ such that
\begin{align*}
  \hat{S} = f(\psignal_1,\ldots,\psignal_n).
\end{align*}

\begin{claim}
  \label{clm:f-properties}
  $f$ satisfies the following conditions.
  \begin{enumerate}
  \item {\bf Symmetry.} For all $x=(x_1,\ldots,x_n) \in \{-1,+1\}^n$ it
    holds that $f(-x_1,\ldots,-x_n)=-f(x_1,\ldots,x_n)$.
  \item {\bf Monotonicity.} $f(x_1,\ldots,x_n)=1$ implies that
    $f(x_1,\ldots,x_{i-1},1,x_{i+1},\ldots,x_n)=1$ for all $i \in [n]$.
  \item {\bf Anonymity.} There exists a subgroup $G \leq S_n$ that
    acts transitively on $[n]$ such that
    $f(x_{\sigma(1)},\ldots,x_{\sigma(n)}) = f(x_1,\ldots,x_n)$ for
    all $x \in \{-1,+1\}^n$ and $\sigma \in G$. 
  \end{enumerate}
\end{claim}
This claim is straightforward to verify, with anonymity a consequence
of the fact that the graph is transitive. 

\subsubsection{Influences, Russo's formula, the KKL theorem and
  Talagrand's theorem}
To prove Theorem~\ref{thm:transitive-learning} we use {\em Russo's
  formula}, a classical result in probability that we prove below.

Let $X_1,\ldots,X_n$ be random variables taking values in
$\{-1,+1\}$. For $-\half < \delta < \half$, let $\mathbb{P}_\delta$ be
the distribution such that $\Psub{\delta}{X_i=+1} = \half + \delta$
independently. Let $g : \{-1,+1\}^n \to \{-1,+1\}$ be a monotone
function (as defined above in Claim~\ref{clm:f-properties}). Let $Y =
g(X)$, where $X = (X_1,\ldots, X_n)$.

Denote by $\tau_i : \{-1,+1\}^n \to \{-1,+1\}^n$ the function given by
$\tau_i(x_1,\ldots,x_n) =
(x_1,\ldots,x_{i-1},-x_i,x_{i+1},\ldots,x_n)$.  We define the {\em
  influence} $I_i^\delta$ of $i \in [n]$ on $Y$ as the probability that $i$
is {\em pivotal}:
\begin{align*}
  I_i^\delta = \Psub{\delta}{g(\tau_i(X)) \neq g(X)}.
\end{align*}
That is $I_i^\delta$ is the probability that the value of $Y=g(X)$ changes,
if we change $X_i$.

\begin{theorem}[Russo's formula]
  \begin{align*}
    \frac{d\Psub{\delta}{Y=+1}}{d\delta} = \sum_iI_i^\delta,    
  \end{align*}
\end{theorem}
\begin{proof}
  Let $\mathbb{P}_{\delta_1,\ldots,\delta_n}$ be the distribution on
  $X$ such that
  \begin{align*}
   \Psub{\delta_1,\ldots,\delta_n}{X_i=+1} = \delta_i.
  \end{align*}
  We prove the claim by showing that
  \begin{align*}
    \frac{\partial\Psub{\delta_1,\ldots,\delta_n}{Y=+1}}{\partial \delta_i} = \Psub{\delta_1,\ldots,\delta_n}{g(\tau_i(X)) \neq g(X)},    
  \end{align*}
  and noting that $\mathbb{P}_{\delta,\ldots,\delta}= \mathbb{P}_{\delta}$, and
  that for general differentiable $h :\R^n \to \R$ it holds that
  \begin{align*}
    \frac{\partial h(y,\ldots,y)}{\partial y} =
    \sum_i\frac{\partial h(x_1,\ldots,x_n)}{\partial x_i}(y).
  \end{align*}

  Indeed, if we denote $\mathbb{E} = \mathbb{E}_{\delta_1,\ldots,\delta_n}$  and
  $\mathbb{P} = \mathbb{P}_{\delta_1,\ldots,\delta_n}$, then
  \begin{align*}
    \frac{\partial}{\partial \delta_i}\P{Y=+1}
    &= \frac{\partial}{\partial \delta_i}\half\E{g(X)}.
  \end{align*}
  Given $x=(x_1,\ldots,x_n)\in \R^n$, denote
  $x_{-i} = (x_1,\ldots,x_{i-1},x_{i+1},\ldots,x_n)$. Then
  \begin{align*}
    \E{g(X)} & =
    \sum_x\P{X_{-i}=x_{-i},X_i=x_i}g(x)\\
    &= \sum_x\P{X_{-i}=x_{-i}}\P{X_i=x_i}g(x),
  \end{align*}
  where the second equality follows from the independence of the
  $X_i$'s. Hence
  \begin{align*}
    \frac{\partial}{\partial \delta_i}\Psub{\delta_1,\ldots,\delta_n}{Y=+1}
    &= \frac{\partial}{\partial \delta_i}\half\sum_x\P{X_{-i}=x_{-i}}\P{X_i=x_i}g(x)\\
    &= \half\sum_x\P{X_{-i}=x_{-i}}x_ig(x),
  \end{align*}
  where the second equality follows from the fact that $\P{X=+1}=\delta_i$
  and $\P{X=-1}=1-\delta_i$. Now, $\sum_{x_i}x_ig(x)$ is equal to zero when
  $g(\tau_i(x))=g(x)$, and to two otherwise, since $g$ is
  monotone. Hence
  \begin{align*}
    \frac{\partial}{\partial \delta_i}\Psub{\delta_1,\ldots,\delta_n}{Y=+1}
    &= \sum_x\P{X_{-i}=x_{-i}}\ind{g(\tau_i(x)) \neq g(x)}\\
    &= \P{g(\tau_i(X)) \neq g(X)}.
  \end{align*}
\end{proof}

Kahn, Kalai and Linial~\cite{KaKaLi:88} prove a deep result on Boolean
functions on the hypercube (i.e., functions from $\{-1,+1\}^n$ to
$\{-1,+1\}$), which was later generalized by
Talagrand~\cite{Talagrand:94}. Their theorem states that there must
exist an $i$ with influence at least $O(\log n/n)$.
\begin{theorem}[Talagrand]
  Let $\eps_\delta = \max_iI_i^\delta$ and $q_\delta =
  \Psub{\delta}{Y=1}$. Then
  \begin{align*}
    \sum_iI_i^\delta \geq 
    K\log \left(1/\eps_\delta\right) q_\delta(1-q_\delta).
  \end{align*}
  for some universal constant $K$.
\end{theorem}

Using this result, the proof of Theorem~\ref{thm:transitive-learning}
is straightforward, and we leave it as an exercise to the reader.

% Using this result, we are ready to prove the main theorem of this
% section.
% \begin{proof}[Proof of Theorem~\ref{thm:transitive-learning}]
%   Recall that $f : \{-1,+1\}^n \to \{-1,+1\}$ given by
%   \begin{align*}
%     \hat{S} = f(\psignal_1,\ldots,\psignal_n)
%   \end{align*}
%   is symmetric, monotone and anonymous.  We condition henceforth on
%   $S=1$ and prove that
%   \begin{align*}
%     \P{\hat{S} \neq 1}{S=1} \leq Cn^{-\frac{C\delta}{\log(1/\delta)}}.
%   \end{align*}
%   Since $f$ is symmetric this will imply the claim.

%   \ocomment{almost done, I think}
% \end{proof}

\chapter{Bayesian Models} 
\label{chap:bayesian}
%\section{Introduction to Bayesian models}
In this chapter we study {\em Bayesian agents}. We call an agent
Bayesian when its actions maximize the expectation of some utility
function. This is a model which comes from Economics, where, in fact,
its use is the default paradigm. We will focus on the case in which an
agent's utility depends only on the state of the world $S$ and on its
actions, and is the same for all agents and all time periods.

\subsection{Toy model: continuous actions}
Before defining general Bayesian models, we consider the following
simple model on an undirected connected graph. Let $S \in \{0,1\}$ be a binary state of the world, and
let the private signals be i.i.d.\ conditioned on $S$.

We denote by $\info^i_t$ the information available to agent $i$ at
time $t$. This includes its private signal, and the actions of its
neighbors in the previous time periods:
\begin{align}
  \label{eq:info-toy}
  \info^i_t = \left\{\psignal_i,\action^j_{t'}\,:\,j\in
    \neigh{i},t'<t\right\}.
\end{align}

The actions are given by
\begin{align}
  \label{eq:action-toy}
  \action^i_t = \CondP{S=1}{\info^i_t}.  
\end{align}
That is, each agent's action is its {\em belief}, or the probability
that it assigns to the event $S=1$, given what it knows.

For this model we prove the following results:
\begin{itemize}
\item {\bf Convergence.} The actions of each agent converge almost
  surely to some $\action^i_\infty$. This is a direct consequence of
  the observation that $\{\sigma(\info^i_t)\}_{t \in \N}$ is a filtration, and
  so $\{\action^i_t\}_{t \in \N}$ is a bounded martingale. Note that this does not use the independence of the signals. 
\item {\bf Agreement.} The limit actions $\action^i_\infty$ are almost
  surely the same for all $i \in V$. This follows from the fact that if $i$ and $j$ are connected then 
  $\action^i_\infty + \action^j_\infty \in \info^i_{\infty} \cap \info^j_{\infty}$ 
   and if $\action^i_\infty$ and $\action^j_\infty$ are not a.s.\ equal then:
\[
\E{\left(\half(\action^i_\infty + \action^j_\infty) - S\right)^2} < \max \big(\E{(\action^i_\infty - S)^2} , \E{(\action^j_\infty - S)^2} \big). 
\] 
Note again that this argument does not use the independence of the
signals.  We will show this in further generality in
Section~\ref{sec:continuous-utils} below. This is a consequence of a
more general agreement theorem that applies to all Bayesian models,
which we prove in Section~\ref{sec:bayesian-agreement}.
\item {\bf Learning.} When $|V|=n$, we show in
  Section~\ref{sec:agreement-to-learning} that $\action^i_\infty =
  \CondP{S=1}{\psignal_1,\ldots,\psignal_n}$. This is the strongest
  possible learning result; the agents' actions are the same as they
  would be if each agent knew all the others' private signals. In
  particular, it follows that $\P{\round{\action^i_\infty} \neq S}$ is
  exponentially small in $n$. This result crucially relies on the independence of the signals as the following example shows. 
\end{itemize}

\begin{example}
  Consider two agents $1,2$ with $\psignal_i = 0$ or $1$ with
  probability $1/2$ each and independently, and $S = \psignal_1 +
  \psignal_2 \mod 2$.  Note that here $\action_i^t = 1/2$ for $i=1,2$
  and all $t$, while it is trivial to recover $S$ from $\psignal_1,
  \psignal_2$.
\end{example} 

\subsection{Definitions and some observations}

Following our general framework (see Section~\ref{sec:definitions}) we
shall (mostly) consider a state of the world $S \in \{0,1\}$ chosen
from the uniform distribution, with conditionally i.i.d.\ private
signals. We will consider both discrete and continuous actions, and
each shall correspond to a different {\em utility function}. A utility
function will simply be a continuous map
$u : \{0,1\} \times [0,1] \to [0,1]$. The quantity $u(S,a)$ represents
what the agent gains when choosing action $a$ when the state is
$S$. In a sense that we will soon define formally, agents will be {\em
  utility maximizers}: they will choose their actions so as to
maximize their utilities.

More precisely, we shall denote by $\util^i_t = u(S,\action^i_t)$
agent $i$'s utility at time $t$, and study {\em myopic agents}, or
agents who strive to maximize, at each period $t$, the expectation of
$\util^i_t$.

As in the toy model above, we denote by $\info^i_t$ the information
available to agent $i$ at time $t$, including its private signal, and
the actions of its neighbors in the previous time periods:
\begin{align}
  \label{eq:info}
  \info^i_t = \left\{\psignal_i,\action^j_{t'}\,:\,j\in
    \neigh{i},t'<t\right\}.
\end{align}
Given a utility function $\util^i_t=u(S,\action^i_t)$, a Bayesian
agent will choose
\begin{align}
  \label{eq:action-max-util}
  \action^i_t = \argmax_s\CondE{u(S,s)}{\info^i_t}.
\end{align}
Equivalently, one can define $\action^i_t$ as a random variable which,
out of all $\sigma(\info^i_t)$-measurable random variables, maximizes
the expected utility:
\begin{align}
  \label{eq:action-max-util2}
  \action^i_t = \argmax_{A \in \sigma(\info^i_t)}\E{u(S,A)}.
\end{align}
We assume that in cases of indifference (i.e., two actions that
maximize the expected utility) the agents chooses one according to
some known deterministic rule.

We consider two utility functions; a discrete one that results in
discrete actions, and a continuous one that results in continuous
actions. The first utility function is
\begin{align}
  \label{eq:discrete-utility}
  \util^i_t = \ind{\action^i_t = S}.
\end{align}
Although this function is not continuous as a function from $[0,1]$ to
$[0,1]$, we will, in this case, consider the set of allowed actions to
be $\{0,1\}$, and so $u : \{0,1\}\times\{0,1\} \to \R$ will be
continuous again.

To maximize the expectation of $\util^i_t$ conditioned on $\info^i_t$,
a myopic agent will choose the action
\begin{align}
  \label{eq:myopic-action-discrete}
  \action^i_t = \argmax_{s \in \{0,1\}}\CondP{S=s}{\info^i_t},
\end{align}
which will take values in $\{0,1\}$. 

We will also consider the following utility function, which
corresponds to continuous actions:
\begin{align}
  \label{eq:cont-utility}
  \util^i_t = 1-\left(\action^i_t - S\right)^2.
\end{align}
To maximize the expectation of this function, an agent will choose the action
\begin{align}
  \label{eq:myopic-action-discrete}
  \action^i_t = \CondP{S=1}{\info^i_t}.
\end{align}
This action will take values in $[0,1]$.

An important concept in the context of Bayesian agents is that of {\em
  belief}. We define agent $i$'s belief at time $t$ to be
\begin{align}
  \label{eq:belief}
  \belief^i_t = \CondP{S=1}{\info^i_t}.
\end{align}
This is the probability that $S=1$, conditioned on all the information
available to $i$ at time $t$. It is easy to check that, in the
discrete action case, the action is the rounding of the belief. In the
continuous action case the action equals the belief.

An important distinction is between {\em bounded} and {\em unbounded}
private signals~\cite{smith2000pathological}. We say that the private
signal $\psignal_i$ is {\bf bounded} when there exists an $\eps>0$
such the private belief $\belief^i_0= \CondP{S=1}{\psignal_i}$ is
supported on $[\eps,1-\eps]$. We will say that it is unbounded when
the private belief $\belief^i_0 = \CondP{S=1}{\psignal_i}$ can be
arbitrarily close to both $1$ and $0$; formally, when the convex
closure of the support of $\belief^i_0$ is equal to $[0,1]$.

Unbounded private signals can be thought of as being ``unboundedly
strong'', and therefore could be expected to promote learning. This is
indeed the case, as we show below.

The following claim follows directly from the fact that the sequence
of sigma-algebras $\sigma(\info^i_t)$ is a filtration.
\begin{claim}
  \label{clm:belief-martingale}
  The sequence of beliefs of agent $i$, $\{\belief^i_t\}_{t \in \N}$,
  is a bounded martingale.
\end{claim}
It follows that a limiting belief almost surely exists, and we can
define
\begin{align}
  \label{eq:lim-belief}
  \belief^i_\infty = \lim_{t \to \infty}\belief^i_t.
\end{align}
Furthermore, if we let $\info^i_\infty = \cup_t\info^i_t$, then
\begin{align}
  \label{eq:lim-knowledge}
  \belief^i_\infty = \CondP{S=1}{\info^i_\infty}.
\end{align}

We would like to also define the limiting {\em action} of agent
$i$. However, it might be the case that the actions of an agent do not
converge. We therefore define $\action^i_t$ to be an {\em action set},
given by the set of accumulation points of the sequence $\action^i_t$.
In the case that $\action^i_\infty$ is a singleton $\{x\}$, we denote
$\action^i_\infty=x$, in a slight abuse of notation. Note that in the
case that actions take values in $\{0,1\}$ (as we will consider
below), $\action^i_\infty$ is either equal to $1$, to $0$, or to
$\{0,1\}$.

The following claim is straightforward.
\begin{claim}
  \label{thm:a-infty-opt}
  Fix a continuous utility function $u$. Then
  \begin{align*}
    \lim_t\CondE{u(S,\action^i_t)}{\info^i_t} =
    \CondE{u(S,a)}{\info^i_\infty} \geq \CondE{u(S,b)}{\info^i_\infty}
  \end{align*}
  for all $a \in \action^i_\infty$ and all $b$.
\end{claim}
That is, any action in $\action^i_\infty$ is optimal (that is,
maximizes the expected utility), given what the agent knows at the
limit $t \to \infty$. It follows that
\begin{align*}
  \CondE{u(S,a)}{\info^i_\infty} = \CondE{u(S,b)}{\info^i_\infty}
\end{align*}
for all $a,b \in \action^i_\infty$.  It also follows that in the case
of actions in $\{0,1\}$, $\action^i_\infty = \{0,1\}$ only if $i$ is
asymptotically indifferent, or expects the same utility from both $0$
and $1$.

We will show that an oft-occurring phenomenon in the Bayesian setting
is agreement on limit actions, so that $\action^i_\infty$ is indeed a
singleton, and $\action^i_\infty = \action^j_\infty$ for all $i,j \in
V$. In this case we can define $\action_\infty$ as the common limit
action.

\section{Agreement}
\label{sec:bayesian-agreement}
In this section we show that regardless of the utility function, and,
in fact, regardless of the private signal structure, Bayesian agents
always reach agreement, except in cases of indifference. This theorem
originated in the work of Aumann~\cite{Aumann:76}, with contributions
by Geanakoplos and others~\cite{geanakoplos1982we,sebenius1983don}. It
first appeared as below in Gale and
Kariv~\cite{GaleKariv:03}. Rosenberg, Solan and
Vieille~\cite{rosenberg2009informational} correct an error in the
proof and extend this result to the even more general setting of {\em
  strategic agents} (which we will not discuss), as is done
in~\cite{mossel2015strategic}.

\begin{theorem}[Gale and Kariv]
  \label{thm:bayesian-agreement}
  Fix a utility function $\util^i_t=u(S,\action^i_t)$, and consider
  $(i, j) \in E$. Then
  \begin{align*}
    \CondE{u(S,a^i)}{\info^i_\infty} = \CondE{u(S,a^j)}{\info^i_\infty} 
  \end{align*}
  for any $a^i \in \action^i_\infty$ and $a^j \in \action^j_\infty$. 
\end{theorem}
That is, any action in $\action^j_\infty$ is optimal, given what $i$
knows, and so has the same expected utility as any action in
$\action^i_\infty$.  Note that this theorem applies even when private
signals are not conditionally i.i.d., and when $S$ is not necessarily
binary.

Note that \eqref{eq:action-max-util2} is a particularly useful way to
think of the agents' actions, as the proof of the following claim
shows.
\begin{claim}
  \label{eq:utils-monotone}
  For all $(i,j) \in E$ it holds that
  \begin{enumerate}
  \item $\E{\util^i_{t+1}} \geq \E{\util^i_t}$.
  \item $\E{\util^i_{t+1}} \geq \E{\util^j_t}$.
  \end{enumerate}
\end{claim}
\begin{proof}
  \begin{enumerate}
  \item Since $\sigma(\info^i_t)$ is included in
    $\sigma(\info^i_{t+1})$, the maximum in
    \eqref{eq:action-max-util2} is taken over a larger space for
    $\action^i_{t+1}$ than it is for $\action^i_t$, and therefore a
    value at least as high is achieved.
  \item Since $\action^j_t$ is $\sigma(\info^i_{t+1})$-measurable, it
    follows from \eqref{eq:action-max-util2} that
    $\E{u(S,\action^i_{t+1})} \geq \E{u(S,\action^j_t)}$.
  \end{enumerate}
\end{proof}

{\bf Exercise.} Prove the following corollary.
\begin{corollary}
  \label{cor:equal-utils}
  For all $i,j \in V$,
  \begin{align*}
    \lim_t\E{\util^i_t} = \lim_t\E{\util^i_j}
  \end{align*}
\end{corollary}

{\bf Exercise.} Prove Theorem~\ref{thm:bayesian-agreement} using
Corollary~\ref{cor:equal-utils} and Claim~\ref{thm:a-infty-opt}.

\section{Continuous utility models}
\label{sec:continuous-utils}
As mentioned above, in the case that the utility function is
\begin{align*}
  \util^i_t = 1-\left(\action^i_t - S\right)^2,
\end{align*}
it follows readily that
\begin{align*}
  \action^i_t = \belief^i_t = \CondP{S=1}{\info^i_t},
\end{align*}
and so, by Claim~\ref{clm:belief-martingale}, the actions of each
agent form a martingale, and furthermore each converge to a singleton
$\action^i_\infty$. Aumann's celebrated {\em Agreement Theorem} from
the paper titled ``Agreeing to Disagree''~\cite{Aumann:76}, as
followed-up by Geanakoplos and Polemarchakis in the paper titled ``We
can't disagree forever''~\cite{geanakoplos1982we}, implies that all
these limiting actions are equal. This follows from
Theorem~\ref{thm:bayesian-agreement}.
\begin{theorem}
  In the continuous utility model
  \begin{align*}
    \action^i_\infty = \CondP{S=1}{\info^i_\infty}
  \end{align*}
  and furthermore
  \begin{align*}
    \action^i_\infty = \action^j_\infty
  \end{align*}
  for all $i,j \in V$.
\end{theorem}
Note again that this holds also for private signals that are not
conditionally i.i.d.
\begin{proof}
  As was mentioned above, since the actions $\action^i_t$ are equal to
  the beliefs $\belief^i_t$, they are a bounded martingale and
  therefore converge. Hence $\action^i_\infty=\belief^i_\infty$ and,
  by \eqref{eq:lim-knowledge},
  \begin{align*}
    \action^i_\infty = \CondP{S=1}{\info^i_\infty}.
  \end{align*}

  Assume $(i,j) \in E$.  By Theorem~\ref{thm:bayesian-agreement} we
  have that
  \begin{align*}
    \CondE{u(S,\action^i_\infty)}{\info^i_\infty} =
    \CondE{u(S,\action^j_\infty)}{\info^i_\infty}.
  \end{align*}
  It hence follows from Claim~\ref{thm:a-infty-opt} that both
  $\action^i_\infty$ and $\action^j_\infty$ maximize
  $\CondE{u(S,\cdot)}{\info^i_\infty}$. But the unique maximizer is
  $\CondP{S=1}{\info^i_\infty}$, and so $\action^i_\infty =
  \action^j_\infty$. For general $i$ and $j$, the claim now follows
  from the fact that the graph is connected.
\end{proof}

\section{Bounds on number of rounds in finite probability spaces}

In this section we consider the case of a {\bf finite} probability
space. Let $S$ be binary, and let the private signals $\overline{\psignal}
= (\psignal_1,\ldots,\psignal_{|V|})$ be chosen from an arbitrary (not
necessarily conditionally independent) distribution over a finite
joint probability space of size $M$. Consider general utility
functions $\util^i_t = u(S,\action^i_t)$.

The following theorem is a strengthening of a theorem by
Geanakoplos~\cite{geanakoplos1994common}, using ideas
from~\cite{MosselTamuz10:arxiv}.
\begin{theorem}[Geanakoplos]
  Let $d$ be the diameter of the graph $G$. Then the actions of each
  agent converge after at most $M \cdot |V|$ time periods:
  \begin{align*}
    \action^i_t = \action^i_{t'}
  \end{align*}
  for all $i \in V$ and all $t,t' \geq M  \cdot |V|$. Furthermore,
  the number of time periods $t$ such that $\action^i_{t+1} \neq
  \action^i_t$ is at most $M$.
\end{theorem}
The key observation is that each sigma-algebra $\sigma(\info^i_t)$
is generated by some subset of the set of random variables
$\left\{\ind{\overline{\psignal}=m}\right\}_{m \in \{1,\ldots,M\}}$.
\begin{proof}
  By \eqref{eq:action-max-util2}, if $\sigma(\info^i_t) =
  \sigma(\info^i_{t'})$ then $\action^i_t = \action^i_{t'}$. It remains to
  show, then, that $\sigma(\info^i_t) = \sigma(\info^i_{t'})$ for all
  $t,t' \geq M \cdot |V|$, and that $\sigma(\info^i_t) \neq
  \sigma(\info^i_{t+1})$ at most $M$ times.

  Now, every sub-sigma-algebra of $\sigma(\overline{\psignal})$ (such
  as $\sigma(\info^i_t)$) is simply a partition of the finite space
  $\{1,\ldots,M\}$. Furthermore, for every $i$, the sequence
  $\sigma(\info^i_t)$ is a filtration, so that each
  $\sigma(\info^i_{t+1})$ is a refinement of $\sigma(\info^i_t)$. A
  simple combinatorial argument shows that any such sequence has at
  most $M$ unique partitions, and so $\sigma(\info^i_t) \neq
  \sigma(\info^i_{t+1})$ at most $M$ times.

  Finally, note that if $\sigma(\info^i_t) = \sigma(\info^i_{t+1})$
  for all $i \in V$ at some time $t$, then this is also the case for
  all later time periods. Hence, as long as the process hasn't ended,
  it must be that $\sigma(\info^i_t) \neq \sigma(\info^i_{t+1})$ for
  some agent $i$. It follows that the process ends after at most $M
  \cdot |V|$ time periods.

\end{proof}

%\section{Gaussian / minimum variance models}
%\ocomment{Maybe we can leave this out? }

\section{From agreement to learning}
\label{sec:agreement-to-learning}
This section is adapted from Mossel, Sly and Tamuz~\cite{mossel2012on}.

In this section we prove two very general results that relate
agreement and learning in Bayesian models. As in our general
framework, we consider a binary state of the world $S \in \{0,1\}$
chosen from the uniform distribution, with conditionally i.i.d.\
private signals. We do not define actions, but only study what can be
said when, at the end of the process (whatever it may be) the agents
reach agreement.

Formally, consider a finite set of agents of size $n$, or a countably
infinite set of agents, each with a private signal $\psignal_i$. Let
$\cF_i$ be the sigma-algebra that represents what is known by agent
$i$. We require that $\psignal_i$ is $\cF_i$ measurable (i.e., each
agent knows its own private signal), and that each $\cF_i$ is a
sub-sigma-algebra of $\sigma(\psignal_1,\psignal_2,\ldots)$. Let agent
$i$'s belief be
\begin{align*}
  \fbelief_i = \CondP{S=1}{\cF_i},
\end{align*}
and let agent $i$'s action be
\begin{align*}
  \action_i = \argmax_{s \in \{0,1\}}\CondP{S=s}{\cF_i}.
\end{align*}
We let $\action_i = \{0,1\}$ when both maximize $\CondP{S=s}{\cF_i}$.

We say that agents agree on beliefs when there exists a random
variable $\fbelief$ such that almost surely $\fbelief_i = \fbelief$
for all agents $i$. Likewise, we say that agents agree on actions when
there exists a random variable $\action$ such that almost surely
$\action_i = \action$ for all agents $i$. Such agreement arises often
as a result of repeated interaction of Bayesian agents.

We show below that agreement on beliefs is a sufficient condition for
learning, and in fact implies the strongest possible type of
learning. We also show that when private signals are unbounded beliefs
then agreement on actions is also a condition for learning.

\subsection{Agreement on beliefs}

The following theorem and its proof is taken
from~\cite{mossel2012on}. This theorem also admits a proof as a
corollary of some well known results on {\em rational expectation
  equilibria} (see,
e.g.,~\cite{demarzo1999uniqueness,ostrovsky2012information}), but we
will not delve into this topic.

\begin{theorem}
  Let the private signals $(\psignal_1,\ldots,\psignal_n)$ be
  independent conditioned on $S$, and let the agents agree on
  beliefs. Then
  \begin{align*}
    \fbelief = \CondP{S=1}{\psignal_1,\ldots,\psignal_n}.
  \end{align*}
\end{theorem}
That is, if the agents have exchanged enough information to agree on
beliefs, they have exchanged {\em all the relevant information}, in
the sense that they have the same belief that they would have had they
shared {\em all the information}.

\begin{proof}
  Denote agent $i$'s {\em private log-likelihood ratio} by 
  \begin{align*}
    \llr_i = \log\frac{d\mu_1^i}{d\mu_0^i}(\psignal_i).
  \end{align*}
  Since $\P{S=1}=\P{S=0}=1/2$ it follows that
  \begin{align*}
    \llr_i = \log\frac{\CondP{S=1}{\psignal_i}}{\CondP{S=0}{\psignal_i}}.
  \end{align*}
  
  Denote $\llr = \sum_{i \in [n]}\llr_i$. Then, since the private signals are
  conditionally independent, it follows by Bayes' rule that
  \begin{align}
    \label{eq:Z-def}
    \CondP{S=1}{\psignal_1,\ldots,\psignal_n} = \logit\left(\llr\right),
  \end{align}
  where $\logit(z) = e^z/(e^z+e^{-z})$.

  Since
  \begin{align*}
    \fbelief=\CondP{S=1}{\fbelief} = \CondE{\CondP{S=1}{\fbelief,\psignal_1,\ldots,\psignal_n}}{\fbelief}
  \end{align*}
  then
  \begin{align}
    \label{eq:x-logit-z}
    \fbelief=\CondE{\logit(\llr)}{\fbelief},  
  \end{align}
  since, given the private signals $(\psignal_1,\ldots,\psignal_n)$,
  further conditioning on $\fbelief$ (which is a function of the private
  signals) does not change the probability of the event $S=1$.

  Our goal is to show that
  $\fbelief=\CondP{S=1}{\psignal_1,\ldots,\psignal_n}$. We will do this by
  showing that conditioned on $\fbelief$, $\llr$ and $\logit(\llr)$ are
  uncorrelated. It will follow that conditioned on $\fbelief$, $\llr$
  is constant, so that $Z = Z(\fbelief)$ and
  \begin{align*}
    \fbelief = \CondP{S=1}{\fbelief} = \CondP{S=1}{Z(\fbelief)} =
    \CondP{S=1}{\psignal_1,\ldots,\psignal_n}.
  \end{align*}

  By the law of total expectation we have that
  \begin{align*}
    \CondE{\llr_i \cdot \logit(\llr)}{\fbelief} = \CondE{\CondE{\llr_i\logit(\llr)}{\fbelief,\llr_i}}{\fbelief}.
  \end{align*}
  Note that $\CondE{\llr_i\cdot\logit(\llr)}{\fbelief,\llr_i} =
  \llr_i\CondE{\logit(\llr)}{\fbelief,\llr_i}$ and so we can write
  \begin{align*}
    \CondE{\llr_i \cdot \logit(\llr)}{\fbelief} = \CondE{\llr_i\CondE{\logit(\llr)}{\fbelief,\llr_i}}{\fbelief}.
  \end{align*}
  Since $\llr_i$ is $\cF_i$ measurable, and since, by
  \eqref{eq:x-logit-z}, $\fbelief = \CondE{\logit(\llr)}{\cF_i} =
  \CondE{\logit(\llr)}{\fbelief}$, then $\fbelief=\CondE{\logit(\llr)}{\fbelief,\llr_i}$
  and so it follows that
  \begin{align}
    \label{eq:z_i-x}
    \CondE{\llr_i \cdot \logit(\llr)}{\fbelief} = \CondE{\llr_i\fbelief}{\fbelief} = \fbelief
    \cdot \CondE{\llr_i}{\fbelief} = \CondE{\logit(\llr)}{\fbelief} \cdot
    \CondE{\llr_i}{\fbelief}.
  \end{align}
  where the last equality is another substitution of \eqref{eq:x-logit-z}.  Summing this
  equation (\ref{eq:z_i-x}) over $i \in [n]$ we get that
  \begin{align}
    \CondE{\llr \cdot \logit(\llr)}{\fbelief} = \CondE{\logit(\llr)}{\fbelief}\CondE{\llr}{\fbelief}.
  \end{align}

  Now, since $\logit(\llr)$ is a monotone function of $\llr$, by
  Chebyshev's sum inequality we have that
  \begin{align}
    \CondE{\llr \cdot \logit(\llr)}{\fbelief} \geq \CondE{\logit(\llr)}{\fbelief}\CondE{\llr}{\fbelief}
  \end{align}
  with equality only if $\llr$ (or, equivalently $\logit(\llr)$) is
  constant. Hence $\llr$ is constant conditioned on $\fbelief$ and the proof
  is concluded. 
  
\end{proof}

\subsection{Agreement on actions}
In this section we consider the case that the agents agree on actions,
rather than beliefs.  The boundedness of private beliefs plays an
important role in the case of agreement on actions. When private
beliefs are bounded then agreement on actions does not imply learning,
as shown by the following example, which is reminiscent of Bala and
Goyal's~\cite{BalaGoyal:96} {\em royal family}. However, when private
beliefs are unbounded then learning does occur with high probability,
as we show below.
\begin{example}
  \label{ex:senate}
  Let there be $n>100$ agents, and call the first hundred ``the
  Senate''. The private signals are bits that are independently equal
  to $S$ with probability $2/3$. Let
  \begin{align*}
    \action_S = \argmax_a\CondP{S=a}{\psignal_1, \ldots,\psignal_{100}},
  \end{align*}
  and let $\cF_i =\sigma(\psignal_i,\action_S)$.
\end{example}

This example describes the case in which the information available to
each agent is the decision of the senate - which aggregates the
senators' private information optimally - and its own private signal.
It is easy to convince oneself that $\action_i=\action_S$ for all $i
\in [n]$, and so actions are indeed agreed upon. However, the
probability that $\action_S\neq S$ - i.e., the Senate makes a mistake
- is constant and does not depend on the number of agents $n$. Hence
the probability that the agents choose the wrong action does {\em not}
tend to zero as $n$ tends to infinity.  This cannot be the case when
private beliefs are unbounded, as Mossel, Sly and
Tamuz~\cite{mossel2012on} show. 
\begin{theorem}[Mossel, Sly and Tamuz]
  \label{thm:common-actions-finite}
  Let the private signals $(\psignal_1,\ldots,\psignal_n)$ be i.i.d.\
  conditioned on $S$, and have unbounded beliefs. Let the agents agree
  on actions. Then there exists a sequence $q(n) = q(n,\mu_0,\mu_1)$,
  depending only on the conditional private signal distributions
  $\mu_1$ and $\mu_0$, such that $q(n) \to 1$ as $n \to \infty$, and
  \begin{align*}
    \P{\action = S} \geq q(n).
  \end{align*}
  In particular,
    \begin{align*}
      q(n) \leq \min_{\eps >
        0}\max\left\{\frac{2\eps}{1-\eps},\frac{4}{n\CondP{B_i <
            \eps}{S=0}}\right\}.
  \end{align*}

\end{theorem}

For the case of a countably infinite set of agent, we prove (using
an essentially identical technique) the following similar statement.

\begin{theorem}
  \label{thm:common-actions}
  Identify the set of agents with $\N$, let the private signals
  $(\psignal_1,\psignal_2,\ldots)$ be i.i.d.\ conditioned on $S$, and
  have unbounded beliefs. Let all but a vanishing fraction of the
  agents agree on actions. That is, let there exist a random variable
  $\action$ such that almost surely
  \begin{align*}
    \limsup_n \frac{1}{n}|\{i \in \N\,:\,\action_i \neq A\}| = 0.
  \end{align*}
  Then $\P{\action = S}=1$.
\end{theorem}

Recall that $\belief^i_0$ denoted the probability of $S=1$ given agent
$i$'s private signal:
\begin{align*}
  \belief^i_0 = \CondP{S=1}{\psignal_i}.
\end{align*}
The condition of unbounded beliefs can be equivalently formulated to
be that for any $\eps>0$ it holds that $\P{\belief^i_0 < \eps} > 0$ and
$\P{\belief^i_0 > 1- \eps} > 0$.

We shall need two standard lemmas to prove this theorem.
\begin{lemma}
  \label{lemma:bu-eps}
  $\CondP{S=0}{\belief^i_0 < \eps} > 1-\eps$.
\end{lemma}
\begin{proof}
  Since $\belief^i_0$ is a function of $\psignal_i$ then
  \begin{align*}
   \CondP{S=1}{\belief^i_0=b_i}=\CondE{\CondP{S=1}{\psignal_i}}{\belief^i_0(\psignal_i)=b_i}
   = \CondE{\belief^i_0}{\belief^i_0=b_i}= b_i,
  \end{align*}
 and so $\CondP{S=1}{\belief^i_0} = \belief^i_0$. It follows that $\CondP{S=0}{\belief^i_0} =
 1-\belief^i_0$, and so $\CondP{S=0}{\belief^i_0 < \eps} > 1-\eps$.
\end{proof}

Lemma~\ref{lemma:conditional-chebyshev} below is a version of Chebyshev's
inequality, quantifying the idea that the expectation of a random
variable $Z$, conditioned on some event $A$, cannot be much lower than
its unconditional expectation when $A$ has high probability.

{\bf Exercise.} Prove the following lemma.
\begin{lemma}
  \label{lemma:conditional-chebyshev}
  Let $Z$ be a real random variable with finite variance,
  and let $A$ be an event. Then
  \begin{align*}
    \E{Z}-\sqrt{\frac{\Var{Z}}{\P{A}}} \leq \CondE{Z}{A} \leq
    \E{Z}+\sqrt{\frac{\Var{Z}}{\P{A}}}
  \end{align*}
\end{lemma}

We are now ready to prove Theorem~\ref{thm:common-actions}.
\begin{proof}[Proof of Theorem~\ref{thm:common-actions}]
  Consider a set of agents $\N$ who agree (except for a vanishing
  fraction) on the action.  Assume by contradiction that $q =
  \CondP{\action \neq 0}{S=0} > 0$.

  Recall that $\fbelief_i = \CondP{S=1}{\cF_i}$. Since
  $\CondP{S=1}{\belief^i_0}=\belief^i_0$, 
  \begin{align*}
   \CondE{\fbelief_i}{\belief^i_0}=\CondE{\CondP{S=1}{\cF_i}}{\belief^i_0}=\CondP{S=1}{\belief^i_0}=\belief^i_0.
  \end{align*}

  Applying Markov's inequality to $\fbelief_i$ we have that $\CondP{\fbelief_i
    \geq \half}{\belief^i_0 < \eps} < 2\eps$, and in particular
  \begin{align*}
    \CondP{\action_i \neq 0,S=0}{\belief^i_0 < \eps} =
    \CondP{\fbelief_i \geq \half,S=0}{\belief^i_0 < \eps} < 2\eps
  \end{align*}
  so
  \begin{align*}
    \P{\action_i \neq 0,S=0,\belief^i_0 < \eps} \leq 2 \eps \P{\belief^i_0 < \eps}   
  \end{align*}

  Denote
  \begin{align}
    \label{K-def}
    K(n) = \frac{1}{n}\sum_{i \in [n]}\ind{\belief^i_0 < \eps} = \frac{1}{n}\sum_{i \in [n]}\ind{\belief^i_0 < \eps, \action_i = 0} + 
    \frac{1}{n}\sum_{i \in [n]}\ind{\belief^i_0 < \eps, \action_i \neq 0}
  \end{align}
  
  Let $K_1(n)$ denote the first sum and $K_2(n)$ denote the second
  sum. From our assumption that a vanishing fraction of agents
  disagree it follows that a.s.
  \begin{align*}
    \lefteqn{\limsup \CondE{K_1(n)}{\action \neq 0,S = 0}}\\
    &\leq \frac{1}{q} \limsup \CondE{K_1(n)}{\action \neq 0}\\
    &\leq 
    \frac{1}{q} \limsup \CondE{\frac{1}{n}\sum_{i \in [n]}\ind{\action_i = 0}}{\action \neq 0} = 0.   
  \end{align*}

  It also follows that for all $n$
  \begin{align*}
\CondE{K_2(n)}{\action \neq 0,S = 0} \leq \frac{1}{q} \E{K_2(n), \action \neq 0, S = 0} \leq \frac{2 \eps \P{\belief^i_0 < \eps}}{q}.    
  \end{align*}
  Thus
  \begin{align*}
    \limsup_n \CondE{K(n)}{\action \neq 0,S = 0} \leq \frac{2 \eps
      \P{\belief^i_0 < \eps}}{q}.
  \end{align*}
   
  We hence bound $\CondE{K}{\action \neq 0,S = 0}$ from above. We will
  now bound it from below to obtain a contradiction.

  Applying lemma~\ref{lemma:conditional-chebyshev} to $K$ and the
  event $\{\action \neq 0\}$ (under the conditional measure $S=0$)
  yields that
  \begin{align*}
    \CondE{K(n)}{\action \neq 0,S=0} &\geq \CondE{K(n)}{S=0} -
    \sqrt{\frac{\Var{K(n) | S = 0}}{q}}.
  \end{align*}
  Since the agents' private signals (and hence their private beliefs)
  are independent conditioned on $S=0$, $K$ (conditioned on $S$) is
  the average of $n$ i.i.d.\ variables. Hence $\Var{K(n) | S = 0} =
  n^{-1}\Var{\ind{\belief^i_0 < \eps} | S = 0}$ and $\CondE{K(n)}{S=0} =
  \CondP{\belief^i_0 < \eps}{S = 0}$. Thus we have that
  \begin{align}
    \label{eq:k_n-upper}
    \CondE{K(n)}{\action \neq 0,S=0} &\geq \CondP{\belief^i_0 < \eps}{S = 0} -
    n^{-1/2}\sqrt{\frac{\Var{\ind{\belief^i_0 < \eps} | S = 0}}{q}}.
  \end{align}
  and so
  \[
  \liminf_n \CondE{K(n)}{\action_i \neq 0,S=0} \geq \CondP{\belief^i_0 < \eps}{S = 0}
  \]
  Joining the lower bound with the upper bound we obtain that
  \[
  \CondP{\belief^i_0 < \eps}{S = 0} \leq \frac{2 \eps \P{\belief^i_0 < \eps}}{q},
  \] 
  and applying Bayes rule we obtain 
 \begin{align*}
    q < \frac{\eps}{\CondP{S=0}{\belief^i_0 < \eps}}.
  \end{align*}

  Since by Lemma~\ref{lemma:bu-eps} above we know that
  $\CondP{S=0}{\belief^i_0 < \eps} > 1-\eps$, then
  \begin{align*}
    q < \frac{\eps}{1-\eps}.
  \end{align*}
  Since this holds for all $\eps$, we have shown that $q=0$, which is
  a contradiction.
\end{proof}

\section{Sequential Models}
In this section we consider a classical class of learning models
called {\bf sequential models}. We retain a binary state of the world
$S$ and conditionally i.i.d.\ private signals, but relax two
assumption.
\begin{itemize}
\item We no longer assume that the graph $G$ is strongly connected. In
  fact, we consider the particular case that the set of agents is
  countably infinite, identify it with $\N$, and let $(i,j) \in E$ iff
  $j < i$. That is, the agents are ordered, and each agent observes
  the actions of its predecessors.
\item We assume that each agent acts once, after observing the actions
  of its predecessors. That is, agent $i$ acts only once, at time
  $i$. 
\end{itemize}
In this section, we denote agent $i$'s (single) action by $\action_i$.
Hence agent $i$'s information when taking its action, which we denote
by $\info_i$, is
\begin{align*}
  \info_i = \{\psignal_i,\action_j\,:\,j < i\}.
\end{align*}
We likewise denote agent $i$'s belief at time $i$ by $\belief_i =
\CondP{S=1}{\info_i}$.  We assume discrete utilities, so that
\begin{align*}
  \action_i = \argmax_{s \in \{0,1\}}\CondP{S=s}{\info_i},
\end{align*}
and let $\action_i=1$ when $\CondP{S=1}{\info_i}=1/2$.

Since each agent acts only once, we explore a different notion of
learning in this section. The question we consider is the following:
when is it the case that $\lim_{i \to \infty}\action_i = S$ with
probability one?  Since the graph is fixed, the answer to this
question depends only on the private signal distributions $\mu_0$ and
$\mu_1$.

This model (in a slightly different form) was introduced independently
by Bikhchandani, Hirshleifer and Welch~\cite{BichHirshWelch:92}, and
Banerjee~\cite{Banerjee:92}. A significant later contribution is that
of Smith and S{\o}rensen~\cite{smith2000pathological}.

An interesting phenomenon that arises in this model is that of an {\em
  information cascade}. An information cascade is said to occur if,
given an agent $i$'s predecessor's actions, $i$'s action does not
depend on its private signal. This happens if the previous agents'
actions present such compelling evidence towards the event that (say)
$S=1$, that any realization of the private signal would not change
this conclusion.  Once this occurs - that is, once one agent's action
does not depend on its private signal - then this will also hold for
all the agents who act later.

\subsection{The external observer at infinity}
An important tool in the analysis of this model is the introduction of
an external observer $x$ that observes all the agents' actions but
none of their private signals. We denote by $\info^x_i =
\{\action_j\,:\,j < i\}$ the information available to $x$ at time $i$,
and denote by
\begin{align*}
 \belief^x_i = \CondP{S=1}{\info^x_i} 
\end{align*}
and
\begin{align*}
 \belief^x_\infty = \lim_i\belief^x_i = \CondP{S=1}{\info^x_\infty} 
\end{align*}
the beliefs of $x$ at times $t$ and infinity respectively, where, as
before, $\info^x_\infty = \cup_i\info^x_i$.
The same martingale argument used above can also be used here to show
that the limit $\belief^x_\infty$ indeed exists and satisfies the
equality above.

{\bf Exercise.} Show that the likelihood ratio
\begin{align*}
  L^x_i = \frac{1-\belief^x_i}{\belief^x_i}
\end{align*}
is also a martingale, {\em conditioned on $S=1$}.

The martingale $\{\belief^x_i\}$ converges almost surely to
$\belief^x_\infty$ in $[0,1]$, and conditioned on $S=1$,
$\belief^x_\infty$ has support $\subseteq (0,1]$. The reason that
$\belief^x_\infty \neq 0$ when conditioning on $S=1$, is the fact that
$\CondP{S=1}{\belief^x_\infty}=\belief^x_\infty$, and so
$\CondP{S=1}{\belief^x_\infty=0}=0$.

We also define actions for $x$, given by
\begin{align*}
  \action^x_i = \argmax_{s \in \{0,1\}}\CondP{S=s}{\info^x_i} = \round{\belief^x_i}.
\end{align*}
We again assume that in cases of indifference, the action $1$ is
chosen.
\begin{claim}
  \label{clm:x-copies}
  $\action^x_{i+1} = \action_i$ 
\end{claim}
That is, the external observer simply copies, at time $t+1$, the
action of agent $t$. This follows immediately from the fact that
$\action_i$ is $\sigma(\info_i)$-measurable, and so $\info^x_{i+1}
\subseteq \info_i$. It follows that $\lim_i\action_i =
\lim_i\action^x_i$, and so we have learning - in the sense we defined
above for this section by $\lim_i\action_i=S$ - iff the external
observer learns in the usual sense of $\lim_i\action^x_i = S$.

\subsection{The agents' calculation}
We write out each agent's calculation of its belief $\belief_i$,
from which follows its action $\action_i$. This is more easily done by
calculating the likelihood ratio
\begin{align*}
  L_i = \frac{1-\belief_i}{\belief_i}.
\end{align*}
By Bayes' law, since $\P{S=1}=\P{S=0}=\half$, and since $\info_i =
(\info^x_i,\psignal_i)$
\begin{align*}
  L_i = \frac{\CondP{S=0}{\info_i}}{\CondP{S=1}{\info_i}}=
  \frac{\CondP{\info_i}{S=0}}{\CondP{\info_i}{S=1}}=
  \frac{\CondP{\info^x_i,\psignal_i}{S=0}}{\CondP{\info^x_i,\psignal_i}{S=1}}.
\end{align*}
Since the private signals are conditionally i.i.d., $\psignal_i$ is
conditionally independent of $\info^x_i$, and so 
\begin{align*}
  L_i &= \frac{\CondP{\info^x_i}{S=0}}{\CondP{\info^x_i}{S=1}}\cdot\frac{\CondP{\psignal_i}{S=0}}{\CondP{\psignal_i}{S=1}}.
\end{align*}
We denote by $P_i$ the {\em private likelihood ratio}
$\CondP{\psignal_i}{S=0} / \CondP{\psignal_i}{S=1}$, so that 
\begin{align}
  \label{eq:L-LxP}
  L_i = L^x_i \cdot P_i.
\end{align}

\subsection{The Markov chain and the martingale}
Another useful observation is that $\{\belief^x_i\}_{i \in \N}$ is not
only a martingale, but also a Markov chain. We denote this Markov
chain on $[0,1]$ by $\cM$. To see this, note that conditioned on $S$,
the private likelihood ratio $P_i$ is independent of $\belief^x_j$,
$j<i$, and so its distribution conditioned on $\belief^x_i =
\CondP{S=1}{\info^x_i}$ is the same as its distribution conditioned on
$(\belief^x_0,\ldots,\belief^x_i)$, which are
$\sigma(\info^x_i)$-measurable.

\subsection{Information cascades, convergence and learning}
An information cascade is the event that, for some $i$, conditioned on
$\info^x_i$, $\action_i$ is independent of $\psignal_i$. That is, an
information cascade is the event that the observer at infinity knows,
at time $i$, which action agent $i$ is going to take, even though it
only knows the actions of $i$'s predecessors and does not know $i$'s
private signal. Equivalently, an information cascade occurs when
$\action_i$ is $\sigma(\info^x_i)$-measurable. It is easy to see that
it follows that $\action_j$ will also be
$\sigma(\info^x_i)$-measurable, for all $j\geq i$.

\begin{claim}
  \label{clm:fixed-point}
  An information cascade is the event that $\belief^x_i$ is a fixed
  point of $\cM$.
\end{claim}
By ``fixed point of $\cM$'' we mean that a.s.\
$\belief^x_{i+1}=\belief^x_{i}$.
\begin{proof}[Proof of Claim~\ref{clm:fixed-point}]
  If $\action_i$ is $\sigma(\info^x_i)$ measurable then
  $\sigma(\info^x_i)=\sigma(\info^x_i,\action_i)=\sigma(\info^x_{i+1})$. It
  follows that
  \begin{align*}
    \belief^x_i = \CondP{S=1}{\info^x_i} = \CondP{S=1}{\info^x_{i+1}}
    = \belief^x_{i+1}.
  \end{align*}

  Conversely, if $\belief^x_i=\belief^x_{i+1}$ w.p.\ one, then
  $\action^x_i=\action^x_{i+1}$ with probability one, and it follows
  that $\action_i=\action^x_{i+1}$ is $\sigma(\info^x_i)$-measurable.
\end{proof}

\begin{theorem}
  The limit $\lim_i\action_i$ exists almost surely.
\end{theorem}
\begin{proof}
  As noted above, $\action_i = \action^x_{i+1}$.  Assume by
  contradiction that $\action^x_{i+1}$ takes both values infinitely
  often. Since $\action^x_i = \ind{\belief^x_i \geq \half}$, and since
  $\belief^x_i$ converges to $\belief^x_\infty$, it follows that
  $\belief^x_\infty = \half$.

  Note that by the Markov chain nature of $\{\belief^x_i\}$,
  \begin{align}
    \label{eq:B-markov-chain}
    \belief^x_{i+1} = f(\belief^x_i,\action_i)
  \end{align}
  for $f:[0,1]\times\{0,1\} \to [0,1]$ independent of $i$ and given by
  \begin{align*}
    f(b,a) = \CondE{\belief_i}{\belief^x_i=b,\action_i=a}.
  \end{align*}
  Since $\action_i = \ind{\belief_i \geq \half}$, it follows that
  $\belief_i = |\belief_i-\half|(2\action_i-1)+\half$, and so
  \begin{align*}
    f(b,a) =
    \CondE{\left|\belief_i-\half\right|}{\belief^x_i=b,\action_i=a}(2a-1)+1/2.
  \end{align*}
  Hence $f$ is continuous at $(1/2,1)$ and $(1/2,0)$, even if
  $\belief_i=\half$ with positive probability. It follows by taking
  the limit of \eqref{eq:B-markov-chain} that if
  $\lim_i\belief^x_i=1/2$ then $f(1/2,1) = f(1/2,0)$. But then
  $\belief^x_i$ would equal $f(1/2,\cdot)$ for all $i$, since
  $\belief^x_0=1/2$, and $\action^x_i=1$ for all $i$, which is a
  contradiction.
\end{proof}

Since $\lim_i\action_i$ exists almost surely we can define
\begin{align*}
  \action = \lim_i\action_i.
\end{align*}
Since $\action_i \neq \action$ for only a finite number of agents, we
can directly apply Theorem~\ref{thm:common-actions} to arrive at the
following result.
\begin{theorem}
  When private signals are unbounded then $\action=S$ w.p.\ one.
\end{theorem}

When private signals are bounded then information cascades occur with
probability one, and $\action$ is no longer almost surely equal to
$S$.
\begin{theorem}
  When private signals are bounded then $\P{\action=S} < 1$.
\end{theorem}
\begin{proof}
  When private signals are bounded then the convex closure of the
  support of $P_i$ is equal to $[\eps,M]$ for some $\eps,M>0$. It
  follows then from \eqref{eq:L-LxP} that if $L^x_i \leq 1/M$ then
  a.s.\ $L_i \leq 1$, and so $\action_i=1$. Likewise, if $L^x_i >
  1/\eps$ then a.s.\ $\action_i=0$. Hence $[0,1/M]$ and
  $(1/\eps,\infty)$ are all fixed points of $\cM$.

  Note that $\CondP{\action^x_i=S}{\info^x_i} =
  \max\{\belief^x_i,1-\belief^x_i\}$. Hence we can prove the claim by
  showing that $\belief^x_\infty = \lim_i\belief^x_i$ is in $(0,1)$,
  since then it would follow that $\lim_i\P{\action^x_i=S} < 1$, and
  in particular $\P{\lim_i\action_i=S}< 1$.

  Indeed, condition on $S=1$, and assume by contradiction that
  $\lim_i\belief^x_i = 1$. Then $L^x_i$ will equal some $\delta \in
  (0,1/M)$ for $i$ large enough. But $\delta$ is a fixed point of
  $\cM$, and so $L^x_j$ will equal $\delta$ hence and $\belief^x_i$
  will not converge to one. The same argument applies if we condition
  on $S=0$ and argue that $L^x_i$ will equal some $N \in
  (1/\eps,\infty)$ for $i$ large enough.
\end{proof}

\section{Learning from discrete actions on networks}
This section is adapted from Mossel, Sly and
Tamuz~\cite{mossel2012asymptotic}.

In this section we study asymptotic learning on general (undirected)
social networks. We here choose to dive more deeply into the proofs -
as compared to the previous sections of this survey - in order to
showcase the various techniques needed to tackle this problem. These
techniques include graph limits (Section~\ref{sec:graph-limits}), a
notion of $\delta$-independence (Section~\ref{sec:delta-ind}) and
more. Indeed, the proof of the main result of this section,
Theorem~\ref{thm:bounded-learning}, does not (as far as we know) admit
a short intuitive explanation, but rather requires the introduction
and digestion of some abstract ideas, and in particular the topology
on rooted graphs that we define and analyze in
Section~\ref{sec:graph-limits}.

We study general social networks that are {\bf undirected}, and
consider both the finite and the countably infinite case.  We consider
agents who maximize, at each time $t$, the utility function (see
\eqref{eq:discrete-utility})
\begin{align*}
  \util^i_t = \ind{\action^i_t = S}.  
\end{align*}
Hence they choose actions using \eqref{eq:myopic-action-discrete}:
\begin{align*}
  \action^i_t = \argmax_{s \in \{0,1\}}\CondP{S=s}{\info^i_t}.
\end{align*}

We ask the following questions:
\begin{enumerate}
\item {\bf Agreement}. Do the agents reach agreement? In this model we
  say that $i$ and $j$ agree if $\action^i_\infty =
  \action^j_\infty$.
  This happens under a weak condition on the private signals.
\item {\bf Learning.} When the agents do agree on some limit action
  $\action_\infty$, does this action equal $S$? The answer to this
  question depends on the graph, and that for undirected graphs indeed
  $\action_\infty = S$ with high probability (for large finite graphs)
  or with probability one (for infinite graphs).
\end{enumerate}

The condition on private signals that implies agreement on limit
actions is the following. By the definition of beliefs,
$\belief^i_0 = \CondP{S=1}{\psignal_i}$. We say that the private
signals {\em induce non-atomic beliefs} when the distribution of
$\belief^i_0$ is non-atomic.  The rational behind this definition is
that it precludes the possibility of {\em indifference} or {\em ties}.
\begin{theorem}
  \label{thm:unbounded-common-knowledge}
  Let $(\mu_0,\mu_1)$ induce non-atomic beliefs. Then there exists a
  random variable $\action_\infty$ such that almost surely
  $\action^i_\infty=\action_\infty$ for all $i$.
\end{theorem}
We refer the reader to~\cite{mossel2012asymptotic} for a proof of this
Theorem. In Section~\ref{app:example-atomic} we give an example that
shows that this claim indeed does not necessarily hold when private
signals are atomic.

The following theorem states that when such agreement is guaranteed
then the agents learn the state of the world with high probability,
when the number of agents is large. This phenomenon is known as {\em
  asymptotic learning}.
\begin{theorem}[Mossel, Sly and Tamuz]
  \label{thm:bounded-learning}
  Let $\mu_0,\mu_1$ be such that for every connected, undirected graph
  $G$ there exists a random variable $\action_\infty$ such that almost
  surely $\action_\infty^i=\action_\infty$ for all $u \in V$.  Then
  there exists a sequence $q(n)=q(n, \mu_0, \mu_1)$ such that $q(n)
  \to 1$ as $n \to \infty$, and $\P{\action_\infty = S} \geq
  q(n)$, for any choice of undirected, connected graph $G$ with $n$
  agents.
\end{theorem}
Informally, when agents agree on limit action sets then they
necessarily learn the correct state of the world, with probability
that approaches one as the number of agents grows. This holds
uniformly over all possible connected and undirected social network
graphs.

The following theorem is a direct consequence of the two theorems
above, since the property proved by
Theorem~\ref{thm:unbounded-common-knowledge} is the condition required
by Theorem~\ref{thm:bounded-learning}.
\begin{theorem}
  \label{thm:non-atomic-learning}
  Let $\mu_0$ and $\mu_1$ induce non-atomic beliefs. Then there exists
  a sequence $q(n)=q(n, \mu_0, \mu_1)$ such that $q(n) \to 1$ as $n
  \to \infty$, and $\P{\action_\infty^i = S} \geq q(n)$, for all agents
  $i$ and for any choice of undirected, connected $G$ with $n$ agents.
\end{theorem}

Before delving into the proof of Theorem~\ref{thm:bounded-learning} we
introduce additional definitions in~\ref{sec:more-defs} and prove some
general lemmas in~\ref{sec:graph-limits},
\ref{sec:isomorphics-balls} and~\ref{sec:delta-ind}.

\subsection{Additional general notation}
\label{sec:more-defs}

We denote the actions of the neighbors of $i$ up to time $t$ by
\begin{align*}
  I^i_t = \{\action^j_{t'}:\:j
  \in \neigh{i},t'<t\},
\end{align*}
and let $I^i_\infty$ denote all the actions of $i$'s neighbors:
\begin{align*}
  I^i_\infty = \{\action^j_{[0,\infty)}:\: j \in \neigh{i}\} = \{\action^j_{t'}:\:j
  \in \neigh{i},t' \geq 0\}.
\end{align*}

We denote the probability that $i$ chooses the correct action at time
$t$ by
\begin{align*}
  p^i_t=\P{\action^i_t=S}.
\end{align*}
and accordingly
\begin{align*}
  p^i_\infty=\lim_{t \to \infty}p^i_t.
\end{align*}

For a set of vertices $U \subseteq V$ we denote by $\psignal(U)$ the
private signals of the agents in $U$.

\subsection{Sequences of rooted graphs and their limits}
\label{sec:graph-limits}
In this section we define a topology on undirected, connected {\em
  rooted graphs}. We call convergence in this topology {\em
  convergence to local limits}, and use it repeatedly in the proof of
Theorem~\ref{thm:bounded-learning}. The core of the proof of
Theorem~\ref{thm:bounded-learning} is the topological
Lemma~\ref{lemma:local_property}, which we prove here. This lemma is a
claim related to {\em local graph properties}, which we also introduce
here.

Let $G=(V,E)$ be an undirected, connected, finite or countably
infinite graph, and let $i \in V$ be a vertex in $G$. We denote by
$(G,i)$ the {\bf rooted graph} $G$ with root $i$.

Let $G=(V,E)$ and $G'=(V',E')$ be graphs. $h : V \to V'$ is a {\bf
  graph isomorphism} between $G$ and $G'$ if $(i,j) \in E
\Leftrightarrow (h(i), h(j)) \in E'$.

Let $(G,i)$ and $(G',i')$ be rooted graphs. Then $h : V \to V'$ is a
{\bf rooted graph isomorphism} between $(G,i)$ and $(G',i')$ if $h$ is
a graph isomorphism and $h(u) = u'$.

We write $(G,i) \cong (G',i')$ whenever there exists a rooted graph
isomorphism between the two rooted graphs.

Given a graph $G=(V,E)$ and two vertices $i, j \in V$, the graph
distance $d(i,j)$ is equal to the length in edges of a shortest
(directed) path between $i$ and $j$.  We denote by $B_r(G, i)$ the
ball of radius $r$ around the vertex $i$ in the graph $G=(V,E)$: Let
$V'$ be the set of vertices $j$ such that $d(i,j)$ is at most $r$. Let
$E' = \{(i,j)\in E:\: i,j \in V'\}$. Then $B_r(G, i)$ is the rooted
graph with vertices $V'$, edges $E'$ and root $i'$.

We next define a topology on (undirected, connected) rooted graphs (or
rather on their isomorphism classes; we shall simply refer to these
classes as graphs).  A natural metric between rooted graphs is the
following (see Benjamini and Schramm~\cite{benjamini2011recurrence},
Aldous and Steele~\cite{aldous2003objective}). Given $(G,i)$ and
$(G',i')$, let
\begin{align*}
  D((G,i),(G',i')) = 2^{-R},
\end{align*}
where
\begin{align*}
  R = \sup \{r : B_r(G,i) \cong B_r(G',i')\}.
\end{align*}
This is indeed a metric: the triangle inequality follows immediately,
and a standard diagonalization argument is needed to show that if
$D((G,i),(G',i'))=0$ then $B_\infty(G,i) \cong B_\infty(G',i')$ and so
$(G,i) \cong (G',i')$.

This metric induces a topology that will be useful to us. As usual,
the basis of this topology is the set of balls of the metric; the ball
of radius $2^{-R}$ around the {\em graph} $(G,i)$ is the set of graphs
$(G',i')$ such that $B_R(G,i) \cong B_R(G',i')$. We refer to
convergence in this topology as convergence to a {\em local limit},
and provide the following equivalent definition for it.

Let $\{(G_r,i_r)\}_{r=1}^\infty$ be a sequence of rooted graphs. We
say that the sequence converges if there exists a rooted graph
$(G',i')$ such that
\begin{align*}
  B_r(G',i') \cong B_r(G_r,i_r),
\end{align*}
for all $r \geq 1$.  We then write
\begin{align*}
  (G',i')=\lim_{r \to \infty}(G_r,i_r),
\end{align*}
and call $(G',i')$ the {\bf local limit} of the sequence
$\{(G_r,i_r)\}_{r=1}^\infty$.

Let $\cG_d$ be the set of rooted graphs with degree at most $d$.

{\bf Exercise.} Show that $\cG_d$ is {\em compact}, and deduce from
that the following lemma:
\begin{lemma}
  \label{lemma:compactness}
  Let $\{(G_r,i_r)\}_{r=1}^\infty$ be a sequence of rooted graphs in
  $\cG_d$. Then there exists a subsequence
  $\{(G_{r_i},i_{r_n})\}_{n=1}^\infty$ with $r_{n+1}>r_n$ for all $n$,
  such that $\lim_{n \to \infty}(G_{r_n},u_{r_n})$ exists.
\end{lemma}

We next define {\em local properties} of rooted graphs.  Let $P$ be
property of rooted graphs or a Boolean predicate on rooted graphs. We
write $(G,i) \in P$ if $(G,i)$ has the property, and $(G,i) \notin P$
otherwise.

We say that $P$ is a {\bf local property} if, for every $(G,i) \in P$
there exists an $r>0$ such that if $B_r(G,i) \cong B_r(G',i')$, then
$(G',i') \in P$. Let $r$ be such that $B_r(G,i) \cong B_r(G',i')
\Rightarrow (G',i') \in P$. Then we say that {\bf $(G,i)$ has property
  $P$ with radius $r$}, and denote $(G,i) \in P^{(r)}$.  That is, if
$(G,i)$ has a local property $P$ then there is some $r$ such that
knowing the ball of radius $r$ around $i$ in $G$ is sufficient to
decide that $(G,i)$ has the property $P$.

An alternative name for a local property would therefore be a {\em
  locally decidable} property. In our topology, local properties are
nothing but {\em open sets}: the definition above states that if
$(G,i) \in P$ then there exists an element of the basis of the
topology that includes $(G,i)$ and is also in $P$. This is a necessary
and sufficient condition for $P$ to be open.

We use this fact to prove the following lemma.  Let $\InfGraphs_d$ be
the set of infinite, connected, undirected graphs of degree at most
$d$, and let $\InfGraphs_d^r$ be the set of $\InfGraphs_d$-rooted
graphs
\begin{align*}
  \InfGraphs_d^r = \{(G,i) \,:\, G \in \InfGraphs_d, i \in G\}.
\end{align*}
{\bf Exercise.} Prove the following lemma.
\begin{lemma}
  \label{lemma:inf-graphs-closed}
  $\InfGraphs_d^r$ is compact.
\end{lemma}

We now state and prove the main lemma of this section.  Note that
the set of graphs $\InfGraphs_d$ satisfies the conditions of this
lemma.
\begin{lemma}
  \label{lemma:local_property}
  Let $\GraphFam$ be a set of infinite, connected graphs, let
  $\GraphFam^r$ be the set of $\GraphFam$-rooted graphs
  \begin{align*}
     \GraphFam^r = \{(G,i) \,:\, G \in \GraphFam, i \in G\},
  \end{align*}
  and assume that $\GraphFam$ is such that $\GraphFam^r$ is compact.

  Let $P$ be a local property such that for each $G \in \GraphFam$
  there exists a vertex $j \in G$ such that $(G,j) \in P$. Then for
  each $G \in \GraphFam$ there exist an $r_0$ and infinitely many
  distinct vertices $\{j_n\}_{n = 1}^\infty$ such that $(G,j_n) \in
  P^{(r_0)}$ for all $n$.
\end{lemma}
\begin{figure}[h]
  \centering
  \includegraphics[scale=1]{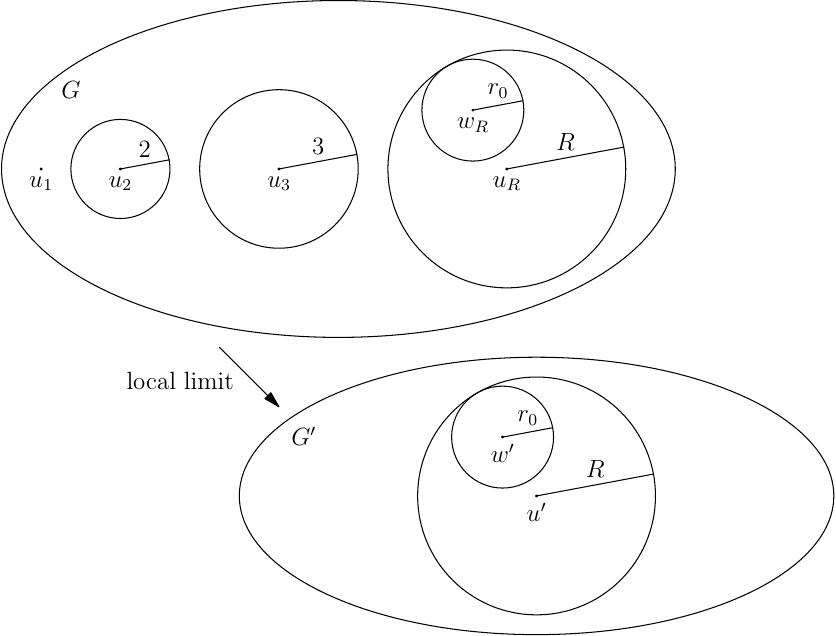}
  \caption{\label{fig:local_property} Schematic diagram of the proof
    of lemma~\ref{lemma:local_property}. The rooted graph $(G',i')$ is
    a local limit of $(G,i_r)$. For $r \geq R$, the ball $B_R(G',i')$
    is isomorphic to the ball $B_R(G,i_r)$, with $w' \in G'$
    corresponding to $j_r \in G$.}
\end{figure}

\begin{proof}
  Let $G$ be an arbitrary graph in $\GraphFam$. Consider a sequence
  $\{k_r\}_{r=1}^\infty$ of vertices in $G$ such that for all $r,s
  \in \N$ the balls $B_r(G,k_r)$ and $B_s(G,k_s)$ are disjoint.

  Since $\GraphFam^r$ is compact, the sequence $\{(G,
  k_r)\}_{r=1}^\infty$ has a converging subsequence $\{(G,
  k_{r_n})\}_{n=1}^\infty$ with $r_{n+1}>r_n$.  Write $i_r = k_{r_n}$,
  and let
  \begin{align*}
    (G',i') = \lim_{r \to \infty}(G, i_r).
  \end{align*}
  Note that since $\GraphFam^r$ is compact, $(G',i') \in \GraphFam^r$
  and in particular $G' \in \GraphFam$ is an infinite,
  connected graph.  Note also that since $r_{n+1}>r_n$, it also holds
  that the balls $B_r(G,i_r)$ and $B_s(G,i_s)$ are disjoint for all
  $r,s \in \N$.

  Since $G' \in \GraphFam$, there exists a vertex $j' \in G'$ such that
  $(G',j') \in P$.  Since $P$ is a local property, $(G',j') \in
  P^{(r_0)}$ for some $r_0$, so that if $B_{r_0}(G',j') \cong
  B_{r_0}(G,j)$ then $(G,j) \in P$.

  Let $R = d(i',j')+r_0$, so that $B_{r_0}(G',j') \subseteq
  B_R(G',i')$. Then, since the sequence $(G,i_r)$ converges to
  $(G',i')$, for all $r \geq R$ it holds that
  $B_R(G,i_r)\cong~B_R(G',i')$. Therefore, for all $r>R$ there exists
  a vertex $j_r \in B_R(G,j_r)$ such that $B_{r_0}(G,j_r) \cong
  B_{r_0}(G',j')$. Hence $(G,j_r) \in P^{(r_0)}$ for all $r>R$ (see
  Fig~\ref{fig:local_property}). Furthermore, for $r,s > R$, the balls
  $B_R(G,i_r)$ and $B_R(G,i_s)$ are disjoint, and so $j_r \neq j_s$.

  We have therefore shown that the vertices $\{j_r\}_{r > R}$ are an
  infinite set of distinct vertices such that $(G,j_r) \in P^{(r_0)}$,
  as required.

\end{proof}

\subsection{Coupling isomorphic balls}
\label{sec:isomorphics-balls}
This section includes technical claims that we will use later. Their
spirit is that everything that happens to an agent up to time $t$
depends only on the state of the world and a ball of radius $t$ around
it. We leave their proofs as an exercise.

\begin{lemma}
  \label{lemma:p-iso}
  Consider two processes with identical private signal distributions
  $(\mu_0,\mu_1)$, on different graphs $G = (V,E)$ and $G' = (V',E')$.

  Let $t \geq 1$, $i \in V$ and $i' \in V'$ be such that there exists
  a rooted graph isomorphism $h:B_t(G,i) \to B_t(G',i')$.

  Let $M$ be a random variable that is measurable in
  $\sigma(\info^i_t)$. Then there exists an $M'$ that is measurable in
  $\info^{i'}_t$ such that the distribution of $(M,S)$ is identical to
  the distribution of $(M',S')$.
\end{lemma}
In particular, we use this lemma in the case where $M$ is an estimator
of $S$. Then this lemma implies that the probability that $M=S$ is
equal to the probability that $M'=S'$.

Recall that $p^i_t = \P{\action^i_t = S} = \max_{A \in
  \sigma(\info^i_t)}\P{A=S}$. Hence we can apply this lemma (\ref{lemma:p-iso})
above to $\action^i_t$ and $\action^{i'}_t$:
\begin{corollary}
  \label{cor:p-iso}
  If $B_t(G,i)$ and $B_t(G',i')$ are isomorphic then $p^i_t =
  p_{i'}(t)$.
\end{corollary}

\subsection{$\delta$-independence}
\label{sec:delta-ind}
To prove that agents learn $S$ we will show that the agents must, over
the duration of this process, gain access to a large number of
measurements of $S$ that are {\em almost} independent. To formalize
the notion of almost-independence we define $\delta$-independence and
prove some easy results about it. The proofs in this section are
again left as an exercise to the reader.

Let $\mu$ and $\nu$ be two measures defined on the same space. We
denote the total variation distance between them by
$\dtv(\mu,\nu)$. Let $A$ and $B$ be two random variables with joint
distribution $\mu_{(A,B)}$. Then we denote by $\mu_A$ the marginal
distribution of $A$, $\mu_B$ the marginal distribution of $B$, and
$\mu_A\times\mu_B$ the product distribution of the marginal distributions.

Let $(X_1,X_2,\dots,X_k)$ be random variables.  We refer to them as
{\bf $\delta$-independent} if their joint distribution
$\mu_{(X_1,\ldots,X_k)}$ has total variation distance of at most
$\delta$ from the product of their marginal distributions
$\mu_{X_1}\times\cdots\times\mu_{X_k}$:
\begin{align*}
  \dtv(\mu_{(X_1,\ldots,X_k)}, \mu_{X_1}\times\cdots\times\mu_{X_k})
  \leq \delta.
\end{align*}
Likewise, $(X_1,\ldots,X_l)$ are {\bf $\delta$-dependent} if the
distance between the distributions is more than $\delta$.

\begin{claim}
  \label{clm:delta-independent}
  Let $A$, $B$ and $C$ be random variables such that $\P{A \neq B} \leq
  \delta$ and $(B,C)$ are $\delta'$-independent. Then $(A,C)$ are
  $2\delta+\delta'$-independent.
\end{claim}

\begin{claim}
  \label{clm:function-independent}
  Let $(X,Y)$ be $\delta$-independent, and let $Z = f(Y,B)$ for some
  function $f$ and $B$ that is independent of both $X$ and $Y$. Then
  $(X,Z)$ are also $\delta$-independent.
\end{claim}

\begin{claim}
  \label{clm:delta-ind-additive}
  Let $A=(A_1,\ldots,A_k)$, and $X$ be random variables. Let
  $(A_1,\ldots,A_k)$ be $\delta_1$-independent and let $(A,X)$ be
  $\delta_2$-independent. Then $(A_1,\ldots,A_k,X)$ are
  $(\delta_1+\delta_2)$-independent.
\end{claim}

As an application of these claim we state the following lemma. The
proof is again left as a (non-trivial) exercise.
\begin{lemma}
  \label{cor:majority}
  For every $1/2 < p < 1$ there exist $\delta = \delta(p) >0$ and
  $\eta = \eta(p) > 0$  such that if $S$ and
  $(X_1,X_2,X_3)$ are binary random variables with $\P{S=1}=1/2$, $1/2
  < p-\eta \leq \P{X_i=S} < 1$, and $(X_1,X_2,X_3)$ are
  $\delta$-independent conditioned on $S$ then $\P{a(X_1,X_2,X_3) = S}
  > p$, where $a$ is the MAP estimator of $S$ given $(X_1,X_2,X_3)$.
\end{lemma}
In other words, one's odds of guessing $S$ using three conditionally
almost-independent bits are greater than using a single bit.

\subsection{Asymptotic learning}
\label{sec:learning}

In this section we prove Theorem~\ref{thm:bounded-learning}.  To prove
this theorem we will need a number of intermediate results, which are
given over the next few sections.

\subsubsection{Estimating the limiting optimal action set $\action_\infty$}
We would like to show that although the agents have a common optimal
action set $\action_\infty$ only at the limit $t \to \infty$, they can estimate
this set well at a large enough time $t$.

The action $\action^i_t$ is agent $i$'s MAP estimator of $S$ at time
$t$. We likewise define $\estL^i_t$ to be
agent $i$'s MAP estimator of $\action_\infty$, at time $t$:
\begin{align}
  \label{eq:l-i-t}
  \estL^i_t = \argmax_{\estL \in 0,1,\{0,1\}\}}
  \CondP{\action_\infty = \estL}{\info^i_t}.
\end{align}
We show that the sequence of random variables $\estL^i_t$ converges to
$\action_\infty$ for every $i$, or that alternatively
$\estL^i_t=\action_\infty$ for each agent $i$ and $t$ large enough:
\begin{lemma}
  \label{lemma:L-estimate}
  $\P{\lim_{t \to \infty}\estL^i_t = \action_\infty} = 1$ for all $i \in V$.
\end{lemma}
Lemma~\ref{lemma:L-estimate} follows by direct application of the more
general Lemma~\ref{lemma:estimation-converges} which we leave as an
exercise.  Note that a consequence is that
$\lim_{t \to \infty} \P{\estL^i_t = \action_\infty} = 1$.

{\bf Exercise.} Prove the following lemma.
\begin{lemma}
  \label{lemma:estimation-converges}
  Let $\cK_1 \subseteq \cK_2, \ldots$ be a filtration of
  $\sigma$-algebras, and let $\cK_\infty = \cup_t\cK_t$. Let $K$ be a
  random variable that takes a finite number of values and is
  measurable in $\cK_\infty$. Let $M(t)=\argmax_k\CondP{K =
    k}{\cK(t)}$ be the MAP estimator of $K$ given $\cK_t$. Then
  \begin{align*}
    \P{\lim_{t \to \infty} M(t)=K} = 1.
  \end{align*}
\end{lemma}

We would like at this point to provide the reader with some more
intuition on $\action^i_t$, $\estL^i_t$ and the difference between
them. Assuming that $\action_\infty = 1$ then by definition, from some
time $t_0$ on, $\action^i_t=1$, and from Lemma~\ref{lemma:L-estimate},
$\estL^i_t=1$. The same applies when $\action_\infty = 0$. However,
when $\action_\infty = \{0,1\}$ then $\action^i_t$ takes both values 0
and 1 infinitely often, but $\estL^i_t$ will eventually equal
$\{0,1\}$. That is, agent $i$ will realize at some point that,
although it thinks at the moment that 1 is preferable to 0 (for
example), it is in fact the most likely outcome that its belief will
converge to $1/2$. In this case, although it is not optimal, a {\em
  uniformly random} guess of which is the best action may not be so
bad. Our next definition is based on this observation.

Based on $\estL^i_t$, we define a second ``action'' $C^i_t$.  Let
$C^i_t$ be picked uniformly from $\estL^i_t$: if $\estL^i_t = 1$ then
$C^i_t = 1$, if $\estL^i_t = 0$ then $C^i_t = 0$, and if $\estL^i_t =
\{0,1\}$ then $C^i_t$ is picked independently from the uniform
distribution over $\{0,1\}$.

Note that we here extend our probability space by including in
$I^i_t$ (the observations of agent $i$ up to time $t$) an extra
uniform bit that is independent of all else and $S$ in
particular. Hence this does not increase $i$'s ability to estimate
$S$, and if we can show that in this setting $i$ learns $S$ then $i$
can also learn $S$ without this bit.  In fact, we show that
asymptotically it is as good an estimate for $S$ as the best estimate
$\action^i_t$:
\begin{claim}
  \label{claim:b-a-equiv}
  \begin{align*}
   \lim_{t \to \infty} \P{C^i_t=S} = \lim_{t \to \infty} \P{\action^i_t=S}
  = p 
  \end{align*}
  for all $i$.
\end{claim}
{\bf Exercise.} Prove Claim~\ref{claim:b-a-equiv}.

\subsubsection{The probability of getting it right}
Recall that $p^i_t = \P{\action^i_t=S}$ and $p^i_\infty = \lim_{t \to
  \infty}p^i_t$ (i.e., $p^i_t$ is the probability that agent $i$ takes
the right action at time $t$). We state here a few easy related claims
that will later be useful to us. The next claim is a rephrasing of the
first part of Claim~\ref{eq:utils-monotone}.
\begin{claim}
  \label{clm:pMonotone}
  $p^i_{t+1} \geq p^i_t$.
\end{claim}
The following claim is a rephrasing of
Corollary~\ref{cor:equal-utils}.
\begin{claim}
  \label{clm:pG}
  There exists a $p \in [0,1]$ such that $p^i_\infty=p$ for all $i$.
\end{claim}

We make the following definition in the spirit of these claims:
\begin{align*}
  p = \lim_{t \to \infty} \P{\action^i_t=S}.  
\end{align*}
In the context of a specific social network graph $G$ we may denote
this quantity as $p(G)$.

For time $t=1$ the next standard claim follows from the fact that the
agents' signals are informative.
\begin{claim}
  \label{clm:pGreaterThanHalf}
  $p^i_t>1/2$ for all $i$ and $t$.
\end{claim}

Recall that $|\neigh{i}|$ is the out-degree of $i$, or the number of
neighbors that $i$ observes. The next lemma states that an agent with
many neighbors will have a good estimate of $S$ already at the second
round, after observing the first action of its neighbors. 
\begin{lemma}
  \label{thm:large-out-deg}
  There exist constants $C_1=C_1(\mu_0,\mu_1)$ and
  $C_2=C_2(\mu_0,\mu_1)$ such that for any agent $i$ it holds that
  \begin{align*}
    p^i_1 \geq 1-C_1 e^{-C_2 \cdot |\neigh{i}|}.
  \end{align*}
\end{lemma}
Intuitively, this follows from the fact that $i$'s neighbors will
provide him with $|\neigh{i}|$ independent signals. We leave the proof
as an exercise.

The following claim is a direct consequence of the previous lemmas of
this section.
\begin{claim}
  \label{thm:large-out-deg-graph}
  Let $d(G) = \sup\{|\neigh{i}|\}$ be the out-degree of the graph $G$;
  note that for infinite graphs it may be that $d(G)=\infty$. Then
  there exist constants $C_1=C_1(\mu_0,\mu_1)$ and
  $C_2=C_2(\mu_0,\mu_1)$ such that
  \begin{align*}
    p(G) \geq 1-C_1 e^{-C_2 \cdot d(G)}.
  \end{align*}
\end{claim}
\begin{proof}
  Let $i$ be an arbitrary vertex in $G$. Then by
  Lemma~\ref{thm:large-out-deg} it holds that
  \begin{align*}
    p^i_1 \geq 1-C_1 e^{-C_2 \cdot \neigh{i}},
  \end{align*}
  for some constants $C_1$ and $C_2$. By Lemma~\ref{clm:pMonotone} we
  have that $p^i_{t+1} \geq p^i_t$, and therefore
  \begin{align*}
    p^i_\infty  = \lim_{n \to \infty}p^i_t \geq 1-C_1 e^{-C_2 \cdot \neigh{i}}.
  \end{align*}
  Finally, $p(G) = p^i_\infty$ by Lemma~\ref{clm:pG}, and so
  \begin{align*}
    p^i_\infty  \geq 1-C_1 e^{-C_2 \cdot \neigh{i}}.
  \end{align*}
  Since this holds for an arbitrary vertex $i$, the claim follows.
\end{proof}

\subsubsection{Local limits and pessimal graphs}
We now turn to apply local limits to our process. We consider here and
henceforth the same model as applied, with the same private signals,
to different graphs. We write $p(G)$ for the value of $p$ on the
process on $G$, $\action_\infty(G)$ for the value of $\action_\infty$
on $G$, etc.

\begin{lemma}
  \label{lemma:limit-leq-p}
  Let $(G,i) = \lim_{r \to \infty}(G_r,i_r)$. Then $p(G) \leq \lim
  \inf_r p(G_r)$.
\end{lemma}
\begin{proof}
  Since $B_r(G_r,i_r) \cong B_r(G,i)$, by Lemma~\ref{cor:p-iso} we
  have that $p^i{r} = p^{i_r}_r$. By Claim~\ref{clm:pMonotone}
  $p^{i_r}_r \leq p(G_r)$, and therefore $p^i_r \leq p(G_r)$. The
  claim follows by taking the limit inferior of both sides.
\end{proof}

Recall that $\InfGraphs_d$ denotes the set of infinite, connected,
undirected graphs of degree at most $d$. Let
\begin{align*}
  \InfGraphs = \bigcup_d \InfGraphs_d.
\end{align*}
Let
\begin{align*}
  p^* = p^*(\mu_0,\mu_1) = \inf_{G \in \InfGraphs}p(G)
\end{align*}
be the probability of learning in the pessimal graph.

Note that by Claim~\ref{clm:pGreaterThanHalf} we have that $p^*>1/2$.
We show that this infimum is in fact attained by some graph:
\begin{lemma}
  \label{lemma:h-exists}
  There exists a graph $H \in \InfGraphs$ such that $p(H)=p^*$.
\end{lemma}
\begin{proof}
  Let $\{G_r = (V_r,E_r)\}_{r=1}^{\infty}$ be a series of graphs in
  $\InfGraphs$ such that $\lim_{r \to \infty} p(G_r) = p^*$. Note that
  $\{G_r\}$ must all be in $\InfGraphs_d$ for some $d$ (i.e., have
  uniformly bounded degrees), since otherwise the sequence $p(G_r)$
  would have values arbitrarily close to $1$ and its limit could not
  be $p^*$ (unless indeed $p^* = 1$, in which case our main
  Theorem~\ref{thm:bounded-learning} is proved). This follows from
  Lemma~\ref{thm:large-out-deg}.

  We now arbitrarily mark a vertex $i_r$ in each graph, so that $i_r
  \in V_r$, and let $(H,i)$ be the limit of some
  subsequence of $\{G_r,i_r\}_{r=1}^\infty$. Since $\InfGraphs_d$ is
  compact (Lemma~\ref{lemma:inf-graphs-closed}), $(H,i)$ is guaranteed
  to exist, and $H \in \InfGraphs_d$.

  By Lemma~\ref{lemma:limit-leq-p} we have that $p(H) \leq \liminf_r
  p(G_r) = p^*$. But since $H \in \InfGraphs$, $p(H)$ cannot be less
  than $p^*$, and the claim is proved.
\end{proof}

\subsubsection{Independent bits}
We now show that on infinite graphs, the private signals in the
neighborhood of agents that are ``far enough away'' are (conditioned
on $S$) almost independent of $\action_\infty$ (the final consensus estimate
of $S$).

\begin{lemma}
  \label{lemma:r-independent}
  Let $G$ be an infinite graph. Fix a vertex $i_0$ in $G$.  Then for
  every $\delta>0$ there exists an $r_\delta$ such that for every
  $r\geq r_\delta$ and every vertex $i$ with $d(i_0,i)>2r$ it holds
  that $\psignal(B_r(G,i))$, the private signals in $B_r(G,i)$, are
  $\delta$-independent of $\action_\infty$, conditioned on $S$.
\end{lemma}
Here we denote graph distance by $d(\cdot,\cdot)$.
\begin{proof}
  Fix $i_0$, and let $i$ be such that $d(i_0,u) > 2r$. Then
  $B_r(G,i_0)$ and $B_r(G,i)$ are disjoint, and hence independent
  conditioned on $S$. Hence $\estL^{i_0}_r$ is independent of
  $\psignal(B_r(G,i))$, conditioned on $S$.

  Lemma~\ref{lemma:L-estimate} states that $\P{\lim_{r \to
      \infty}\estL^{i_0}_r = \action_\infty} = 1$, and so there
  exists an $r_\delta$ such that for every $r \geq r_\delta$ it holds
  that $\P{\estL^{i_0}_r = \action_\infty} > 1-\half\delta$.

  Recall Claim~\ref{clm:delta-independent}: for any $A,B,C$, if
  $\P{A=B} = 1-\half\delta$ and $B$ is independent of $C$, then
  $(A,C)$ are $\delta$-independent.

  Applying Claim~\ref{clm:delta-independent} to $\action_\infty$,
  $\estL^{i_0}_r$ and $\psignal(B_r(G,i))$ we get that for any $r$
  greater than $r_\delta$ it holds that $\psignal(B_r(G,i))$ is
  $\delta$-independent of $\action_\infty$, conditioned on $S$.
\end{proof}

We will now show, in the lemmas below, that in infinite graphs each
agent has access to any number of ``good estimators'':
$\delta$-independent measurements of $S$ that are each almost as
likely to equal $S$ as $p^*$, the minimal probability of estimating
$S$ on any infinite graph.

We say that agent $i \in G$ has {\bf $k$ $(\delta,\eps)$-good
  estimators} if there exists a time $t$ and estimators
$M_1,\ldots,M_k$ such that $(M_1,\ldots,M_k) \in \info^i_t$ and
\begin{enumerate}
\item $\P{M_i = S} > p^* - \eps$ for $1 \leq i \leq k$.
\item $(M_1, \ldots, M_k)$ are $\delta$-independent, conditioned on
  $S$.
\end{enumerate}

The proof of the next claim is straightforward.
\begin{claim}
  \label{clm:good-est-local}
  Let $P$ denote the property of having $k$ $(\delta,\eps)$-good
  estimators. Then $P$ is a {\em local property} of the rooted graph
  $(G,i)$. Furthermore, if $u \in G$ has $k$ $(\delta,\eps)$-good
  estimators measurable in $\info^i_t$ then $(G,i) \in P^{(t)}$, i.e.,
  $(G,i)$ has property $P$ with radius $t$.
\end{claim}

We are now ready to prove the main lemma of this subsection:
\begin{lemma}
  \label{thm:independent-bits}
  For every $d \geq 2$, $G \in \InfGraphs_d$, $\eps, \delta > 0$ and
  $k \geq 0$ there exists a vertex $i$, such that $i$ has $k$
  $(\delta,\eps)$-good estimators.
\end{lemma}
Informally, this lemma states that if $G$ is an infinite graph with
bounded degrees, then there exists an agent that eventually has $k$
almost-independent estimates of $S$ with quality close to $p^*$, the
minimal probability of learning.

\begin{proof}
  In this proof we use the term ``independent'' to mean ``independent
  conditioned on $S$''.

  We choose an arbitrary $d$ and prove by induction on $k$. The basis
  $k = 0$ is trivial.  Assume the claim holds for $k$, any $G \in
  \InfGraphs_d$ and all $\eps,\delta>0$. We shall show that it holds
  for $k+1$, any $G \in \InfGraphs_d$ and any $\delta,\eps>0$.

  By the inductive hypothesis for every $G \in \InfGraphs_d$ there
  exists a vertex in $G$ that has $k$ $(\delta/100,\eps)$-good
  estimators $(M_1, \ldots, M_k)$.

  Now, having $k$ $(\delta/100,\eps)$-good estimators is a {\em local
    property} (Claim~\ref{clm:good-est-local}).  We now therefore
  apply Lemma~\ref{lemma:local_property}: since every graph $G \in
  \InfGraphs_d$ has a vertex with $k$ $(\delta/100,\eps)$-good
  estimators, any graph $G \in \InfGraphs_d$ has a time $t_k$ for
  which infinitely many distinct vertices $\{j_r\}$ have $k$
  $(\delta/100,\eps)$-good estimators measurable at time $t_k$.

  In particular, if we fix an arbitrary $i_0 \in G$ then for every $r$
  there exists a vertex $j \in G$ that has $k$
  $(\delta/100,\eps)$-good estimators and whose distance $d(i_0,j)$
  from $i_0$ is larger than $r$.

  We shall prove the lemma by showing that for a vertex $j$ that is
  far enough from $i_0$ which has $(\delta/100,\eps)$-good estimators
  $(M_1,\ldots,M_k)$, it holds that for a time $t_{k+1}$ large enough
  $(M_1,\ldots,M_k,C^j_{t_{k+1}})$ are $(\delta,\eps)$-good estimators.

  By Lemma~\ref{lemma:r-independent} there exists an $r_\delta$ such
  that if $r>r_\delta$ and $d(i_0,j) > 2r$ then $\psignal(B_r(G,j))$
  is $\delta/100$-independent of $\action_\infty$.  Let $r^* =
  \max\{r_\delta,t_k\}$, where $t_k$ is such that there are infinitely
  many vertices in $G$ with $k$ good estimators measurable at time
  $t_k$.

  Let $j$ be a vertex with $k$ $(\delta/100,\eps)$-good estimators
  $(M_1,\ldots,M_k)$ at time $t_k$, such that $d(i_0,j) > 2r^*$.
  Denote
  \begin{align*}
    \bar{M}=(M_1, \ldots, M_k).
  \end{align*}
  Since $d(i_0,j) > 2r_\delta$, $\psignal(B_{r^*}(G,j))$ is
  $\delta/100$-independent of $\action_\infty$, and since
  $B_{t_k}(G,j) \subseteq B_{r^*}(G,j)$, $\psignal(B_{t_k}(G,j))$ is
  $\delta/100$-independent of $\action_\infty$. Finally, since
  $\bar{M} \in \sigma(\info^j_{t_k})$, $\bar{M}$ is a function of
  $\psignal(B_{t_k}(G,j))$, and so by
  Claim~\ref{clm:function-independent} we have that $\bar{M}$ is also
  $\delta/100$-independent of $\action_\infty$.

  For $t_{k+1}$ large enough it holds that
  \begin{itemize}
  \item $\estL^j_{t_{k+1}}$ is equal to $\action_\infty$ with
    probability at least $1-\delta/100$, since
    \begin{align*}
      \lim_{t \to \infty} \P{\estL^j_t=\action_\infty} = 1,
    \end{align*}
     by Claim~\ref{lemma:L-estimate}.
  \item Additionally, $\P{C^j_{t_{k+1}}=S} > p^*-\eps$, since
    \begin{align*}
      \lim_{t \to \infty}\P{C^j_t=S} = p \geq p^*,
    \end{align*}
    by Claim~\ref{claim:b-a-equiv}.
  \end{itemize}

  We have then that $(\bar{M},\action_\infty)$ are $\delta/100$-independent and
  $\P{\estL^j_{t_{k+1}} \neq \action_\infty} \leq \delta/100$.
  Claim~\ref{clm:delta-independent} states that if $(A,B)$ are
  $\delta$-independent $\P{B \neq C} \leq \delta'$ then $(A,C)$ are
  $\delta+2\delta'$-independent. Applying this here we get that
  $(\bar{M},\estL^j_{t_{k+1}})$ are $\delta/25$-independent.

  It follows by application of Claim~\ref{clm:delta-ind-additive} that
  $(M_1,\ldots,M_k,\estL^j{t_{k+1}})$ are $\delta$-independent.  Since
  $C^j_{t_{k+1}}$ is a function of $\estL^j_{t_{k+1}}$ and an
  independent bit, it follows by another application of
  Claim~\ref{clm:function-independent} that $(M_1, \ldots, M_k,
  C^j_{t_{k+1}})$ are also $\delta$-independent.

  Finally, since $\P{C^j_{t_{k+1}}=S} > p^*-\eps$, $j$ has the $k+1$
  $(\delta,\eps)$-good estimators $(M_1,\ldots,C^j_{t_{k+1}})$ and the
  proof is concluded.

\end{proof}
\subsubsection{Asymptotic learning}
As a tool in the analysis of finite graphs, we would like to prove
that in infinite graphs the agents learn the correct state of the
world almost surely.

\begin{theorem}
  \label{thm:bounded}
  Let $G=(V,E)$ be an infinite, connected undirected graph with
  bounded degrees (i.e., $G$ is a general graph in $\InfGraphs$). Then
  $p(G)=1$.
\end{theorem}
Note that an alternative phrasing of this theorem is that $p^*=1$.
\begin{proof}
  Assume the contrary, i.e. $p^*<1$. Let $H$ be an infinite, connected
  graph with bounded degrees such that $p(H) = p^*$, such as we have
  shown exists in Lemma~\ref{lemma:h-exists}.

  By Lemma~\ref{thm:independent-bits} there exists for arbitrarily
  small $\eps,\delta >0 $ a vertex $w \in H$ that has access at some
  time $T$ to three $\delta$-independent estimators (conditioned on
  $S$), each of which is equal to $S$ with probability at least
  $p^*-\eps$. By Claims~\ref{cor:majority}
  and~\ref{clm:pGreaterThanHalf}, the MAP estimator of $S$ using these
  estimators equals $S$ with probability higher than $p^*$, for the
  appropriate choice of low enough $\eps,\delta$. Therefore, since
  $j$'s action $\action^j_t$ is the MAP estimator of $S$, its
  probability of equaling $S$ is $\P{\action^j_t=S} > p^*$ as well,
  and so $p(H) > p^*$ - contradiction.
\end{proof}

Using Theorem~\ref{thm:bounded} we can now prove
Theorem~\ref{thm:bounded-learning}, which is the corresponding
theorem for finite graphs:
\begin{proof}[Proof of Theorem~\ref{thm:bounded-learning}]
  Assume the contrary. Then there exists a series of graphs $\{G_r\}$
  with $r$ agents such that $\lim_{r \to \infty}\P{\action_\infty(G_r)=S} < 1$,
  and so also $\lim_{r \to \infty}p(G_r) < 1$.

  By the same argument of Theorem~\ref{thm:bounded} these graphs must
  all be in $\InfGraphs_d$ for some $d$, since otherwise, by
  Lemma~\ref{thm:large-out-deg-graph}, there would exist a subsequence
  of graphs $\{G_{r_d}\}$ with degree at least $d$ and $\lim_{d \to
    \infty}p(G_{r_d}) = 1$. Since $\InfGraphs_d$ is compact
  (Lemma~\ref{lemma:inf-graphs-closed}), there exists a graph $(G,i)
  \in \InfGraphs_d$ that is the limit of a subsequence of $\{(G_r,
  i_r)\}_{r=1}^\infty$.

  Since $G$ is infinite and of bounded degree, it follows by
  Theorem~\ref{thm:bounded} that $p(G) = 1$, and in particular
  $\lim_{r \to \infty}p^i_\infty(r) = 1$. As before, $p_{i_r}(r) =
  p^i_\infty(r)$, and therefore $\lim_{r \to \infty}p_{i_r}(r) =
  1$. Since $p(G_r) \geq p_{i_r}(r)$, $\lim_{r \to \infty}p(G_r) = 1$,
  which is a contradiction.
\end{proof}

\subsection{Example of Non-atomic private beliefs leading to
  non-learning}
\label{app:example-atomic}

We sketch an example in which private beliefs are atomic and
asymptotic learning does not occur.

\begin{example}
  \label{example:non-learning}
  Let the graph $G$ be the undirected chain of length $n$, so that
  $V=\{1, \ldots, n\}$ and $(i,j)$ is an edge if $|i-j| = 1$. Let the
  private signals be bits that are each independently equal to $S$
  with probability $2/3$. We choose here the tie breaking rule under
  which agents defer to their original signals\footnote{We conjecture
    that changing the tie-breaking rule does not produce asymptotic
    learning, even for randomized tie-breaking.}.
\end{example}

We leave the following claim as an exercise to the reader.
\begin{claim}
  If an agent $i$ has at least one neighbor with the same private
  signal (i.e., $\psignal_i=\psignal_j$ for $j$ a neighbor of $i$)
  then $i$ will always take the same action $\action^i_t =
  \psignal_i$.
\end{claim}
Since this happens with probability that is independent of $n$, with
probability bounded away from zero an agent will always take the wrong
action, and so asymptotic learning does not occur. It is also clear
that optimal action sets do not become common knowledge, and these
fact are indeed related.

\bibliographystyle{abbrv} \bibliography{survey}

\end{document}